\documentclass[12pt,a4paper,twoside,reqno]{amsart}
\title[Gravitating vortices and the Einstein--Bogomol'nyi equations]
{Gravitating vortices and the\\ Einstein--Bogomol'nyi equations}

\author[L. \'Alvarez-C\'onsul]{Luis \'Alvarez-C\'onsul}
\address{Instituto de Ciencias Matem\'aticas (CSIC-UAM-UC3M-UCM)\\ Nicol\'as Cabrera 13--15, Cantoblanco\\ 28049 Madrid, Spain}
  \email{l.alvarez-consul@icmat.es}
\author[M. Garcia-Fernandez]{Mario Garcia-Fernandez}
\address{Dep. Matem\'aticas\\ Universidad Aut\'onoma de Madrid\\ and
  Instituto de Ciencias Matem\'aticas (CSIC-UAM-UC3M-UCM)\\ Ciudad
  Universitaria de Cantoblanco\\ 28049 Madrid, Spain}
\email{mario.garcia@icmat.es}
\author[O. Garc\'{\i}a-Prada]{Oscar Garc\'{\i}a-Prada}
\address{Instituto de Ciencias Matem\'aticas (CSIC-UAM-UC3M-UCM)\\ Nicol\'as Cabrera 13--15, Cantoblanco\\ 28049 Madrid, Spain}
  \email{oscar.garcia-prada@icmat.es}
\author[V. P. Pingali]{Vamsi Pritham Pingali}
\address{Department of Mathematics, Indian Institute of Science, Bangalore, India - 560012}
\email{vamsipingali@iisc.ac.in}

\thanks{The work of the first three authors was partially supported by the Spanish MINECO under ICMAT Severo Ochoa project No. SEV-2015-0554, and under grant No. MTM2013-43963-P. The work of the second author was partially supported by the Nigel Hitchin Laboratory under the ICMAT Severo Ochoa grant. The research leading to these results has received funding from the European Union's Horizon 2020 Programme (H2020-MSCA-IF-2014) under grant agreement No. 655162, and by the European Commission Marie Curie IRSES MODULI Programme PIRSES-GA-2013-612534. The work of the fourth author was partially supported by SERB grant No. ECR/2016/001356. The fourth author is grateful to ICMAT for hosting him under the MODULI programme. The fourth author also thanks the Infosys foundation for the Infosys young investigator grant.}

\usepackage{hyperref,url}
\usepackage{amssymb,latexsym}
\usepackage{amsmath,amsthm}
\usepackage[latin1]{inputenc}
\usepackage{enumerate}
\usepackage[all]{xy} \CompileMatrices
\SelectTips{cm}{12} 

\usepackage{verbatim} 

\usepackage{wrapfig}
\usepackage{graphicx}

\def\Yint#1{\mathchoice
    {\YYint\displaystyle\textstyle{#1}}%
    {\YYint\textstyle\scriptstyle{#1}}%
    {\YYint\scriptstyle\scriptscriptstyle{#1}}%
    {\YYint\scriptscriptstyle\scriptscriptstyle{#1}}%
      \!\int}
\def\YYint#1#2#3{{\setbox0=\hbox{$#1{#2#3}{\int}$}
    \vcenter{\hbox{$#2#3$}}\kern-.52\wd0}}
\def\fint{\Yint-}

\theoremstyle{plain}
\newtheorem{theorem}{Theorem}[section]
\newtheorem{lemma}[theorem]{Lemma}
\newtheorem{corollary}[theorem]{Corollary}
\newtheorem{proposition}[theorem]{Proposition}
\newtheorem{conjecture}[theorem]{Conjecture}
\newtheorem*{theorem*}{Theorem}
\theoremstyle{definition}
\newtheorem{definition}[theorem]{Definition}
\newtheorem{definition-theorem}[theorem]{Definition-Theorem}
\newtheorem{example}[theorem]{Example}

\theoremstyle{remark}
\newtheorem{remark}[theorem]{Remark}

\numberwithin{equation}{section} 

\newcommand{\pr}{p}
\newcommand{\Id}{\operatorname{Id}}

\newcommand{\Aut}{\operatorname{Aut}}
\newcommand{\Ext}{\operatorname{Ext}}
\newcommand{\dbar}{\bar{\partial}}

\newcommand{\CC}{{\mathbb C}}
\newcommand{\PP}{{\mathbb P}}

\newcommand{\RR}{{\mathbb R}}
\newcommand{\ZZ}{{\mathbb Z}}

\renewcommand{\(}{\left(}
\renewcommand{\)}{\right)}
\newcommand{\vol}{\operatorname{vol}}
\newcommand{\Vol}{\operatorname{Vol}}
\newcommand{\defeq}{\mathrel{\mathop:}=} 

\newcommand{\surj}{\to\kern-1.8ex\to}

\newcommand{\lto}{\longrightarrow}
\newcommand{\lra}[1]{\stackrel{#1}{\longrightarrow}}

\newcommand{\cA}{\mathcal{A}}

\newcommand{\cJ}{\mathcal{J}}

\newcommand{\cM}{\mathcal{M}}

\newcommand{\cF}{\mathcal{F}}

\newcommand{\cG}{\mathcal{G}}

\newcommand{\cO}{\mathcal{O}}

\newcommand{\cT}{{\mathcal{T}}}

\newcommand{\Lie}{\operatorname{Lie}}

\newcommand{\LieG}{\operatorname{Lie} \cG}
\newcommand{\cX}{{\widetilde{\mathcal{G}}}}
\newcommand{\LieX}{\operatorname{Lie} \cX}

\newcommand{\cH}{\mathcal{H}} 
\newcommand{\LieH}{\Lie\cH}
\newcommand{\chM}{\widehat{\mathcal{M}}}

\newcommand{\GL}{\operatorname{GL}}
\newcommand{\SL}{\operatorname{SL}}
\newcommand{\U}{\operatorname{U}}
\newcommand{\SU}{\operatorname{SU}}
\newcommand{\Diff}{\operatorname{Diff}}

\newcommand{\R}{{\mathbb{R}}}

\begin{document}

\begin{abstract}
In this work we consider the gravitating vortex equations. These equations couple a metric over a compact Riemann surface with a hermitian metric over a
holomorphic line bundle equipped with a fixed global section --- the
Higgs field ---, and have a symplectic interpretation as moment-map
equations. As a particular case of the gravitating vortex equations on $\PP^1$, we
find the Einstein--Bogomol'nyi equations, previously studied in the
theory of cosmic strings in physics. We prove two main results in this
paper. Our first main result gives a converse to an existence theorem
of Y. Yang for the Einstein--Bogomol'nyi equations, establishing in
this way a correspondence with Geometric Invariant Theory for these
equations.  In particular, we prove a conjecture by Y. Yang about the
non-existence of cosmic strings on $\PP^1$ superimposed at a single
point. Our second main result is an existence and uniqueness result
for the gravitating vortex equations in genus greater than one.
\end{abstract}

\maketitle

\setlength{\parskip}{5pt}
\setlength{\parindent}{0pt}

\tableofcontents

\section{Introduction}\label{sec:intro}


Let $\Sigma$ be a compact Riemann surface. Let $L$ be a holomorphic line bundle over $\Sigma$ and $\phi\in H^0(\Sigma,L)$ a holomorphic global section of $L$.
For constant parameters $\alpha,\tau \in \RR$, the gravitating vortex equations are
\begin{equation}\label{eq:gravvortexeq}
\begin{split}
i\Lambda F + \frac{1}{2}(|\phi|^2-\tau) & = 0,\\
S + \alpha(\Delta + \tau) (|\phi|^2 -\tau) & = c.
\end{split}
\end{equation}
They involve two unknowns: a K\"ahler metric $g_\Sigma$ on
$\Sigma$ and a hermitian metric $h$ on $L$. Here, $F$ is the curvature
of the Chern connection of $h$, $\Lambda F$ is its contraction by the
K\"ahler form $\omega$ of $g_\Sigma$, $|\phi|$ is the pointwise norm
of $\phi$ with respect to $h$, $S$ is the scalar curvature of
$g_\Sigma$, and $\Delta$ is the Laplacian of the metric on the surface
acting on functions. The constant $c\in\RR$ is topological, as it can
be obtained by integrating~\eqref{eq:gravvortexeq} over
$\Sigma$. Explicitly, it is
\begin{equation}\label{eq:constantcintro}
c=\frac{2\pi(\chi(\Sigma)-2\alpha\tau c_1(L))}{\Vol_\omega(\Sigma)}.
\end{equation}
The gravitating vortex equations were obtained in \cite{AGG2} as a dimensional reduction of the K\"ahler--Yang--Mills equations on a complex surface and have a moment-map interpretation (see Section \ref{sec:mmap}). Being a particular case of the K\"ahler-Yang--Mills equations \cite{AGG}, the coupled system \eqref{eq:gravvortexeq} is motivated by the fundamental question  of understanding moduli spaces for algebraic varieties equipped with vector bundles, as proposed by S.-T. Yau \cite{Yau}. Throughout this paper, we will assume $\alpha \geq 0$ and $\tau > 0$.

Solutions of the first equation in \eqref{eq:gravvortexeq}, known as
the vortex equation (also known as Bogomol'nyi equations in the
abelian Higgs model) are called vortices, and have been extensively
studied in the literature after the seminal work of Jaffe and
Taubes~\cite{Jaffe-Taubes,Taubes1} on the Euclidean plane, and
Witten~\cite{Witten} on the 2-dimensional Minkowski spacetime. It is
known~\cite{Brad,G1,Noguchi} that the existence of solutions (with
$\phi\neq 0$) is equivalent to the inequality
\begin{equation}\label{eq:ineqtau}
c_1(L) < \frac{\tau \Vol_\omega(\Sigma)}{4\pi}.
\end{equation}
As proved by the third author \cite{G3}, this result follows the same
principles as the Theorem of Donaldson, Uhlenbeck and Yau~\cite{D3,UY}
--- relating the existence of solutions of the Hermitian--Yang--Mills
equations with an algebraic numerical condition ---, and indeed can be
obtained as a corollary of this correspondence \cite{G1}.
Since the scalar equation in~\eqref{eq:gravvortexeq} couples the
vortices to a Riemannian metric on $\Sigma$, it seems reasonable to
refer to the solutions of the system~\eqref{eq:gravvortexeq} as
gravitating vortices. In fact, this system ties up with the physics of
cosmic strings when $c=0$ and $c_1(L)>0$ 
(see~\cite{AGG2}
for background). In this case, it becomes equivalent to the
Einstein--Bogomol'nyi equations on a Riemann
surface~\cite{Yang1992,Yang1994} and, as observed by
Yang~\cite{Yang,Yang3}, if they have solutions, then our assumption
that $\Sigma$ is compact implies that it is the Riemann sphere
$\PP^1$; see Section~\ref{sub:Einstein-Bogomonyi} for details.

The core of this paper
(Sections~\ref{sec:futaki}--\ref{sec:higher-genus}) is
devoted to address questions related to the existence and uniqueness
of solutions of~\eqref{eq:gravvortexeq}. 
Let us temporarily skip Sections~\ref{sec:futaki}--\ref{sec:geodesics}, and start with the main result in higher genus (see Section~\ref{sec:higher-genus} for details).

\begin{theorem}
\label{thm:higher-genus.intro}
Let $\Sigma$ be a compact Riemann surface of genus $g\geq 2$, and $L$
a holomorphic line bundle over $\Sigma$ of degree $N>0$ equipped with a
holomorphic section $\phi \neq 0$. Let $\tau$ be a real constant such
that $0<N<\tau/2$. Define
\begin{equation}
\label{eq:critical-alpha.intro}
\alpha_*\defeq\frac{2g-2}{2\tau(\tau/2-N)}>0.  
\end{equation}
Then, the set of $\alpha$ for which~\eqref{eq:gravvortexeq} has
smooth solutions of volume $2\pi$ is open and contains the closed
interval $[0,\alpha_*]$. Furthermore, the solution is unique for
$\alpha\in[0,\alpha_*]$.
\end{theorem}

This can be compared with the classical result which establishes that
a compact Riemann surface admits a metric of constant curvature with
fixed volume, unique up to biholomorphisms. The proof of
Theorem~\ref{thm:higher-genus.intro} involves the continuity method,
where openness is proven using the moment-map interpretation given
in~Section~\ref{sec:mmap}, while closedness needs \emph{a priori}
estimates as usual. 
The hardest part is the $C^0$ estimate, and in fact it is for this
estimate that the value of $\alpha$ should not be too large. With
these estimates at hand, we prove uniqueness by adapting an argument
by Bando and Mabuchi in the K\"ahler--Einstein situation~\cite{BM}.
An interesting open question is to
see what the largest value of $\alpha$ is, for which solutions
exist. Notice that in the \emph{dissolving limit} $\tau \to N/2$ of the vortex we have $\phi \to 0$ (see \cite{G1}), and $\alpha_*$ in \eqref{eq:critical-alpha.intro} becomes arbitrarily large.

Turning now to the case of surfaces of lower genus, we observe that in
genus $g=1$, the gravitating vortex equations~\eqref{eq:gravvortexeq}
(with $\phi\neq 0$) always have a solution in the weak coupling limit
$0< \alpha \ll 1$ (see~\cite[Theorem~4.1]{AGG2} for a
precise formulation), and it is an interesting open problem to find
effective bounds for $\alpha$ for which~\eqref{eq:gravvortexeq} admit
solutions.

Unlike the cases of genus $g\geq 1$, we show in
Sections~\ref{sec:futaki}--\ref{sec:geodesics} that in
genus $g=0$, new phenomena arise that did not appear in the classical
situation of constant curvature metrics on a surface, namely there exist obstructions to the existence of solutions of \eqref{eq:gravvortexeq}. This may be interpreted
as saying that our problem is comparatively closer to the more
sophisticated problem of Calabi on the existence of K\"ahler--Einstein
metrics, where algebro-geometric stability obstructions appear on
compact K\"ahler manifolds with $c_1>0$.

Based on our moment-map interpretation of the gravitating vortex
equations~\eqref{eq:gravvortexeq}, in genus $g=0$ we pursue an
analogue for them of the theorem of Donaldson, Uhlenbeck and
Yau~\cite{D3,UY}, in the case $\alpha>0$.  The first clue pointing out
to such a correspondence for $\Sigma = \PP^1$ lies in Yang's existence
result~\cite{Yang,Yang3}, reformulated more elegantly in the language
of Mumford's Geometric Invariant Theory (GIT)~\cite{MFK} (it is
perhaps worth emphasizing that physicists did not have a moment-map
interpretation of the Einstein--Bogomol'nyi equations, and this result
was not formulated in the language of GIT; see \cite{AGG2} for
details).

\begin{theorem}[Yang's existence theorem]\label{th:Yangintro}
Suppose $c=0$, $\alpha > 0$ and that~\eqref{eq:ineqtau} is satisfied. Let $D=\sum n_jp_j$
be an effective divisor on $\PP^1$ corresponding to a pair $(L,\phi)$.
Then, the Einstein--Bogomol'nyi equations on $(\PP^1,L,\phi)$ have
solutions, provided that the divisor $D$ is GIT polystable for the
canonical linearized $\textup{SL}(2,\CC)$-action on the space of effective
divisors.
\end{theorem}

Our main result in genus $g=0$ is the following converse to Theorem~\ref{th:Yangintro}, for the more general gravitating vortex
equations (see Section~\ref{subsec:geodesics}).

\begin{theorem}\label{th:Yangconjectureintro}
If $(\PP^1,L,\phi)$ admits a solution of the gravitating vortex
equations with $\alpha > 0$, then \eqref{eq:ineqtau} holds and the
divisor $D$ is polystable for the $\textup{SL}(2,\CC)$-action.
\end{theorem}

The key idea for its proof comes from  the observation that the powerful methods
of the theory of symplectic and GIT quotients are ideally suited to
analyze the gravitating vortex equations. Using the general theory for coupled
equations developed in \cite{AGG}, in Sections
\ref{sec:futaki}--\ref{sec:geodesics} we construct obstructions to the existence of solutions for the gravitating vortex equations on $\PP^1$, and apply them to prove Theorem \ref{th:Yangconjectureintro}. Observe that, even though the equations \eqref{eq:gravvortexeq} depend on the section $\phi \in H^0(\Sigma,L)$, our obstruction, alike Yang's existence criterion, only depends on the line spanned by $\phi$. 

Combining now Theorems~\ref{th:Yangintro}
and~\ref{th:Yangconjectureintro}, we obtain a 
correspondence theorem for the Einstein--Bogomol'nyi equations.

\begin{theorem}\label{th:HK}
A triple $(\PP^1,L,\phi)$ with $\phi\neq 0$
admits a solution of the Einstein--Bogomol'nyi equations with $\alpha > 0$ if and only
\eqref{eq:ineqtau} holds and the divisor $D$ is polystable for the
$\textup{SL}(2,\CC)$-action. 
\end{theorem}

Note that Theorem~\ref{th:HK} does not claim uniqueness of solutions
modulo automorphisms of $(\PP^1,L,\phi)$. However this should be
expected on general grounds, as our methods rely on appropriate
versions of techniques developed over the years for the recently
solved K\"ahler--Einstein problem~\cite{ChDoSun} (see also
\cite{BBJ,DS,CSW,Tian}).  In fact, Theorem~\ref{th:HK} can be seen as
a 2-dimensional toy model for this problem. As in this case,
uniqueness is a delicate issue, related to the geodesic equation in
the space of K\"ahler potentials~\cite{D6,Mab1,Se}; see
Section~\ref{sec:conjecturemoduli} for details and a discussion of the
relevance of the homogeneous complex Monge--Amp\`ere equation in this
problem. 

Theorem~\ref{th:HK} clarifies Yang's guess~\cite{Yang3} that the
location of the zeros of $\phi$ should ``play an important role to
global existence'' and his observation that the condition corresponding to strict polystability is a ``borderline situation'' (with solutions preserved by an $S^1$-action). By comparison with partial results for the
case $D=Np$ (where he proved solutions cannot be
$S^1$-symmetric~\cite[Theorem~1.1(ii)]{Yang3}), he stated the
following (see also \cite[p. 437]{Yangbook}).

\begin{conjecture}[Yang's conjecture]\label{conj:Yangintro}
There is no solution of the Einstein--Bogomol'nyi equations for $N$
strings superimposed at a single point, that is, when $D=Np$.
\end{conjecture}

Combining Theorem \ref{th:Yangconjectureintro} with Kirwan's result on the stability of effective divisors on $\mathbb{P}^1$ \cite{MFK}, we settle Conjecture~\ref{conj:Yangintro} in the affirmative. It is worth mentioning that the proof of Theorem \ref{th:Yangconjectureintro} is done in two steps, excluding first in Theorem \ref{th:Yangconjecture.2} the unstable divisors supported at one or two points, by construction and direct evaluation of a Futaki invariant for the gravitating vortex equations.

Theorem \ref{thm:higher-genus.intro} and Theorem
\ref{th:Yangconjectureintro} establish an interesting link between our
theory for the gravitating vortex equations and the classical
Teichm\"uller
spaces. 
Ever since the pioneering work of Fricke and Teichm\"uller, it has
been convenient to enhance the geometric data of the Riemann surface
with marked points or other types of decoration. In the spirit of
\cite{Yau2005}, this paper paves the way to extend the complex
analytic approach to these moduli spaces to one of the simplest
decorations for which this space has not been so well studied yet,
namely an effective divisor on the surface (cf., e.g.,~\cite[\S
2.1.3]{Hassett}, for the corresponding algebro-geometric moduli
problem). As a matter of fact, Theorem \ref{thm:higher-genus.intro}
directly implies that in genus $g > 1$, and for an effective range of
the parameter $\alpha$, a suitable moduli space of solutions to
\eqref{eq:gravvortexeq} can be interpreted as a Teichm\"uller space
for pairs consisting of a compact Riemann surface and an effective
divisor. This aspect of our theory is explored and made precise in
Section \ref{sub:moduli}.

Further motivation for this work comes the theory of cosmic strings in physics. Even though the physical relevance of the equations \eqref{eq:gravvortexeq} has only been established for $c = 0$ \cite{Yang1992}, we expect that our methods may apply to other cosmic string models in the literature with non-zero cosmological constant (see Remark \ref{rem:cosmologicalconstant}). In particular, our analysis of the positive genus case in Section \ref{sec:higher-genus} may be useful for the models with $\Lambda < 0$ considered in \cite{Maheda}, while the methods of Theorem \ref{th:Yangconjectureintro} may provide new insight in the non-abelian situation considered in \cite{Yang1994CMP}, for the unexplored case of compact surfaces.

The gravitating vortex equations were obtained in~\cite{AGG2} as a dimensional reduction of the K\"ahler--Yang--Mills equations introduced in~\cite{AGG}, and
in Theorem \ref{thm:dim-red} we use Theorem~\ref{thm:higher-genus.intro} to construct a new class
of non-trivial solutions of the K\"ahler--Yang--Mills equations.

\section{The gravitating vortex equations}\label{sec:gravvort}

In this section we introduce the gravitating vortex equations 
and reformulate Yang's Existence Theorem 
in the language of Geometric Invariant Theory.

\subsection{Gravitating vortices}\label{subsec:gravvort}

Let $\Sigma$ be a compact connected Riemann surface of arbitrary
genus, $L$ a holomorphic line bundle over $\Sigma$, and $\phi$ a
global holomorphic section of $L$. Fix real constants $\tau > 0$ and $\alpha \geq 0$, respectively called the \emph{symmetry breaking parameter} and the
\emph{coupling constant}.

\begin{definition}\label{def:gravvorteq}
The \emph{gravitating vortex equations}, for a K\"ahler metric on
$\Sigma$ with K\"ahler form $\omega$ and a hermitian metric $h$ on
$L$, are
\begin{equation}\label{eq:gravvortexeq1}
\begin{split}
i\Lambda_\omega F_h + \frac{1}{2}(|\phi|_h^2-\tau) & = 0,\\
S_\omega + \alpha(\Delta_\omega + \tau) (|\phi|_h^2 -\tau) & = c.
\end{split}
\end{equation}
\end{definition}

In~\eqref{eq:gravvortexeq1}, $F_h$ is the curvature 2-form of the
Chern connection of $h$, $\Lambda_\omega F_h\in C^\infty(\Sigma)$ is
its contraction with $\omega$, $|\phi|_h\in C^\infty(\Sigma)$ is the
pointwise norm of $\phi$ with respect to $h$, $S_\omega$ is the scalar
curvature of $\omega$ (as usual, K\"ahler metrics will be identified
with their associated K\"ahler forms), and $\Delta_\omega$ is the
Laplace operator for the metric $\omega$, given by
\[
\Delta_\omega f = 2i \Lambda_\omega\dbar\partial f,
\]
for $f\in C^\infty(\Sigma)$. The constant $c \in \RR$ is topological,
and is explicitly given by
\begin{equation}\label{eq:constantc}
c = \frac{2\pi(\chi(\Sigma) - 2\alpha\tau c_1(L))}{\Vol_\omega(\Sigma)},
\end{equation}
with $\Vol_\omega(\Sigma)\defeq\int_\Sigma\omega$, as can be deduced
by integrating~\eqref{eq:gravvortexeq1} over $\Sigma$.

Given a fixed K\"ahler metric $\omega$, the first equation in~\eqref{eq:gravvortexeq1}, that is,
\begin{equation}\label{eq:vortexeq}
i\Lambda_\omega F_h + \frac{1}{2}(|\phi|_h^2-\tau) = 0
\end{equation}
is the \emph{vortex equation} for a hermitian metric $h$ on $L$, also known as the Bogomol'nyi equations in the abelian-Higgs model. The solutions of~\eqref{eq:vortexeq} are called 
\emph{vortices} 
and correspond to the absolute minima of an energy functional \cite{Brad,G3}.

%

If $\phi=0$, then the existence of solutions of \eqref{eq:vortexeq} is equivalent
by the Hodge Theorem to the numerical condition $c_1(L)=\tau\Vol_\omega(\Sigma)/4\pi$.
For $\phi\neq 0$, Noguchi~\cite{Noguchi}, Bradlow~\cite{Brad} and the third
author~\cite{G1,G3} gave independently and with different methods the
following characterization of the existence of vortices.

\begin{theorem}
\label{th:B-GP}
Assume that $\phi\neq 0$. Then, for every fixed K\"ahler form
$\omega$, there exists a solution $h$ of the vortex
equation~\eqref{eq:vortexeq} if and only if
\begin{equation}\label{eq:ineq}
c_1(L) < \frac{\tau \Vol_\omega(\Sigma)}{4\pi},
\end{equation}
in which case the solution is unique.
\end{theorem}

When $\alpha >0$, finding a solution of the vortex equation \eqref{eq:vortexeq} is not enough to solve the more complicated system of equations in Definition \ref{def:gravvorteq}. 
As explained in Section \ref{sec:intro}, the gravitating vortex equations \eqref{eq:gravvortexeq1} describe vortices on a Riemann surface coupled with the K\"ahler metric $\omega$, as the second equation in \eqref{eq:gravvortexeq1} intertwines the scalar curvature of $\omega$ with the function $|\phi|^2_h$. The \emph{gravitating vortices}, that is, the solutions of~\eqref{eq:gravvortexeq1}, are the main subject of the present paper. An important goal of our study, partially achieved in Theorem \ref{th:HK}, is to find an analogue of Theorem \eqref{th:B-GP} for gravitating vortices.

We discuss next two simple examples, where it is straightforward to characterize the existence of gravitating vortices with $\alpha > 0$.

\begin{example}\label{ex:phi0}
If $\phi=0$, then the existence of solutions of~\eqref{eq:vortexeq} is equivalent
to the numerical condition $c_1(L)=\tau\Vol_\omega(\Sigma)/4\pi$. This
follows from the fact that for $\phi=0$, the first and second equations
in~\eqref{eq:gravvortexeq1} reduce to the conditions that $h$ is a
Hermite--Einstein metric on $L$ and $\omega$ is a constant scalar
curvature K\"ahler metric on $\Sigma$, respectively. Therefore, the
equivalence is a consequence of the Hodge Theorem applied to the
equation in $u\in C^\infty(\Sigma)$ obtained from~\eqref{eq:vortexeq}
by making a conformal transformation $h'=e^{2u}h$ to a fixed $h$, and
the Uniformization Theorem for Riemann surfaces.
\end{example}

\begin{example}\label{ex:c10}
If $\phi \neq 0$ and $c_1(L) = 0$, the gravitating vortex equations~\eqref{eq:gravvortexeq1} always have
a solution for $\tau>0$. This follows from the fact that if $c_1(L)=0$ and $H^0(\Sigma,L)\neq 0$, then $L\cong\cO_\Sigma$ (see e.g.~\cite[Ch. IV]{Ha}). By Theorem \ref{th:B-GP}, for any choice of K\"ahler metric $\omega$ on $\Sigma$, the unique solution of \eqref{eq:vortexeq} is the constant hermitian metric $h$ on the trivial line bundle $L$, satisfying $|\phi|_h^2=\tau$. We conclude that the unique solutions of \eqref{eq:gravvortexeq1} in this case are pairs $(\omega,h)$ such that $h$ is constant, $|\phi|_h^2=\tau$, and $\omega$ has constant scalar curvature.
\end{example}

The conditions $\phi \neq 0$ and $c_1(L) > 0$ will be assumed throughout the rest of this paper. 

The sign of $c$ plays an important role in the problem of existence of gravitating vortices. The dependence of the gravitating vortex equations~\eqref{eq:cosmicstrings} on the topological constant $c$ is better observed using the following K\"ahler--Einstein type formulation of~\eqref{eq:gravvortexeq1}, where $\rho_\omega$ is the Ricci form of $\omega$:
\begin{equation}\label{eq:KEtype}
\begin{split}
i\Lambda_\omega F_h + \frac{1}{2}(|\phi|_h^2-\tau) & = 0,\\
\rho_\omega - \alpha dd^c(|\phi|_h^2) - 2\alpha\tau i F_h & = c \omega.
\end{split}
\end{equation}
Using that $\Sigma$ is compact, this system reduces to a second-order system of PDE. To see this, we fix a constant scalar curvature metric $\omega_0$ on $\Sigma$ and the unique hermitian metric $h_0$ on $L$ with constant $\Lambda_{\omega_0} F_{h_0}$, and apply a conformal change to $h$ while changing $\omega$ within its K\"ahler class. The equations \eqref{eq:KEtype} for $\omega = \omega_0 + dd^c v, h=e^{2f}h_0$, with $v,f\in C^\infty(\Sigma)$, are equivalent to
the following semi-linear system of partial differential equations (cf.  \cite[Lemma~4.3]{AGG2})
\begin{equation}\label{eq:KWtype0}
\begin{split}
\Delta f + \frac{1}{2}(e^{2f}|\phi|^2-\tau)e^{4\alpha \tau f - 2 \alpha e^{2f}|\phi|^2 - 2 c v} & = - c_1(L),\\
\Delta v + e^{4\alpha \tau f - 2 \alpha e^{2f}|\phi|^2 - 2 c v} & = 1.
\end{split}
\end{equation} 
Here, $\Delta$ is the Laplacian of the fixed metric $\omega_0$, normalized to have volume $2\pi$ and $|\phi|$ is the pointwise norm with respect to the fixed metric $h_0$ on $L$. Note that $\omega = (1- \Delta v) \omega_0$ implies $1 - \Delta v >0$, which is compatible with the last equation in \eqref{eq:KWtype0}. 

For $c\geq 0$, the existence of gravitating vortices forces the topology of the surface to be that of the $2$-sphere, because  $c_1(L)>0$ implies $\chi(\Sigma) > 0$ by~\eqref{eq:constantc}. The important case $c=0$, for which the system \eqref{eq:KWtype0} reduces to a single PDE, is treated in Section \ref{sub:Einstein-Bogomonyi}. The genus zero case of the gravitating vortex equations is studied in Section  \ref{sec:futaki} and Section \ref{sec:geodesics}. The condition $c < 0$ has important consequences for the analysis of \eqref{eq:KWtype0}, and is considered in genus $\geqslant 2$ in Section \ref{sec:higher-genus}.

\subsection{The Einstein-Bogomol'nyi equations}
\label{sub:Einstein-Bogomonyi}

When $c$ in~\eqref{eq:constantc} is zero, the gravitating vortex
equations~\eqref{eq:gravvortexeq1} turn out to be a system of partial
differential equations that have been extensively studied in the
physics literature. Following Yang~\cite{Yang1992,Yang1994CMP}, we will refer to them as
the \emph{Einstein--Bogomol'nyi equations}:
\begin{equation}\label{eq:cosmicstrings}
\begin{split}
i\Lambda_\omega F_h + \frac{1}{2}(|\phi|_h^2-\tau) & = 0,\\
S_\omega + \alpha(\Delta_\omega + \tau) (|\phi|_h^2 -\tau) & = 0.
\end{split}
\end{equation}
As observed by Yang~\cite[Section 1.2.1]{Yang}, the existence of solutions
of~\eqref{eq:cosmicstrings} with $\alpha>0$ 
constrains the topology of $\Sigma$ to be the complex projective
line (or 2-sphere) $\PP^1$, since $c=0$ if and only if
\[
\chi(\Sigma) = 2\alpha\tau c_1(L)
\]
(recall that we are assuming $\tau > 0$ and $c_1(L) > 0$).

\begin{remark}\label{rem:cosmologicalconstant}
In \cite{Yang1994CMP} the equations \eqref{eq:cosmicstrings} are derived from the Abelian Higgs model coupled with gravity in four dimensions with cosmological constant $\Lambda$. Under the \emph{string ansatz} for the four-dimensional Lorentz metric, the Einstein field equations imply \eqref{eq:gravvortexeq1} with $c = \Lambda$, jointly with the condition $\Lambda g = T$ (see \cite[Equation (2.4)]{Yang1994CMP}),
where $T$ denotes the components of the stress-energy along the Riemann surface. Here we assume that $\Sigma$ is compact or assymptotically euclidean \cite{Yang1992}. The additional constraint $g^{jk}T_{jk} = 0$ requires that $\Lambda = 0$. Based on this observation, we expect that the general theory for the gravitating vortex equations developed in this work can be extended to other cosmic string equations in the physics literature, which allow for a non-vanishing cosmological constant. Examples are provided in \cite[Sections 3-4]{Yang1994CMP} and \cite{Maheda}. 
\end{remark}


The particular features of the Einstein--Bogomol'nyi equations~\eqref{eq:cosmicstrings} are better observed using the K\"ahler--Einstein type formulation of the gravitating vortex
equations~\eqref{eq:gravvortexeq1}, given by \eqref{eq:KWtype0}. In the case $c=0$, for  $L=\mathcal{O}_{\PP^1}(N)$ and 
$$
e^{2u} = 1- \Delta v
$$
the system \eqref{eq:KWtype0} reduces to a single partial differential equation
\begin{equation}\label{eq:single}
\begin{split}
\Delta f + \frac{1}{2}e^{2u}(e^{2f}|\phi|^2-\tau) & = - N,
\end{split}
\end{equation}
for a function $f\in C^\infty(\PP^1)$, where
\[
u=2\alpha\tau f-\alpha e^{2f}|\phi|^2+c',
\] 
and $c'$ is a real constant that can be chosen at will. By studying the Liouville type equation
\eqref{eq:single} on $\PP^1$, Yang \cite{Yang,Yang3} proved the
existence of solutions of the Einstein--Bogomol'nyi equations under
certain numerical conditions on the zeros of $\phi$, to which he refers as ``technical restriction''~\cite[Section 1.3]{Yang}. It turns out
that these conditions have a precise algebro-geometric meaning in the
context of Mumford's Geometric Invariant Theory (GIT)~\cite{MFK}, as a
consequence of the following result by Kirwan.

\begin{proposition}[{\cite[Ch.~4, Proposition~4.1]{MFK}}]\label{prop:GIT}
Consider the space of effective divisors 
on $\PP^1$ with its canonical linearised $\SL(2,\CC)$-action. Let
$D=\sum_j n_jp_j$ be 
an effective divisor, for finitely many different points $p_j\in\PP^1$
and integers $n_j>0$ such that $N=\sum_j n_j$. Then
\begin{enumerate}
\item[\textup{(1)}] $D$ is stable if and only if $n_j < \frac{N}{2}$ for all $j$.
\item[\textup{(2)}] $D$ is strictly polystable if and only if $D=\frac{N}{2}p_1 + \frac{N}{2}p_2$, where $p_1 \neq p_2$ and $N$ is even.
\item[\textup{(3)}] $D$ is unstable if and only if there exists $p_j
  \in D$ such that $n_j>\frac{N}{2}$.
\end{enumerate}
\end{proposition}

Using Proposition~\ref{prop:GIT}, Yang's existence theorem has the
following reformulation, where ``GIT polystable'' means either
conditions (1) or (2) of Proposition~\ref{prop:GIT} are satisfied, and
\[
D = \sum_j n_j p_j
\]
is the effective divisor on $\PP^1$ corresponding to a pair
$(L,\phi)$, with $N = \sum_j n_j = c_1(L)$.

\begin{theorem}[Yang's Existence Theorem]\label{th:Yang}
Assume that $\alpha > 0$ and that \eqref{eq:ineq} holds. Then, there exists a solution of the Einstein--Bogomol'nyi equations~\eqref{eq:cosmicstrings} on $(\PP^1,L,\phi)$ if $D$ is GIT polystable for the 
linearised $\operatorname{SL}(2,\CC)$-action of the space of effective divisors.
\end{theorem}

For the benefit of the reader, we comment briefly on the proof. If $D$ is stable, then the existence of solutions of the Einstein--Bogomol'nyi equations follows by Yang's
result~\cite[Theorem~1.2]{Yang} and part (1) of
Proposition~\ref{prop:GIT}.
Yang also proved~\cite[Theorem~1.1(i)]{Yang3} that the
Einstein--Bogomol'nyi equations have a solution if
$D=\frac{N}{2}p+\frac{N}{2}\overline{p}$, for $N$ even and antipodal
points $p,\overline{p}$ on $\PP^1$, and that this solution admits an
$S^1$-symmetry given by rotation along the $\{p,\overline{p}\}$ axis.
If $D$ is an arbitrary strictly polystable effective divisor, so $D=
\frac{N}{2}p_1+\frac{N}{2}p_2$ as in part (2) of
Proposition~\ref{prop:GIT}, then the existence of solutions of the
Einstein--Bogomol'nyi equations follows from Yang's result, by pulling
back his solution by an element of $\SL(2,\CC)$ mapping $p_1,p_2$
to a pair of antipodal points.

As mentioned in Section \ref{sec:intro}, regarding obstructions to the existence of solutions
of~\eqref{eq:cosmicstrings} Yang formulated Conjecture \ref{conj:Yangintro} in~\cite[Section 6, p. 590]{Yang3} (later stated as an open problem~\cite[p. 437]{Yangbook}), which corresponds to the case of unstable effective divisor $D = Np$. The proof of Yang's Conjecture will be addressed in Section \ref{sec:futakievaluation}.

\section{The symplectic origin of the gravitating vortex equations}\label{sec:mmap}

The gravitating vortex equations were first obtained~\cite{AGG2} by
dimensional reduction of the K\"ahler--Yang--Mills
equations~\cite{AGG,GF}, whereby the solutions acquired an
interpretation as the zeros of a moment map in the general theory of
symplectic quotients, for suitable infinite-dimensional manifolds. A
direct approach to this moment-map interpretation, as described in
this section, is however better suited to prove obstructions for the
existence of gravitating vortices in the next sections (it will rely
on previous calculations~\cite{D7,GFR}).

\subsection{A hamiltonian action on the space of sections of a line bundle}
\label{sub:Ham-action}

Let $S$ be a compact connected oriented smooth surface and $L$ a
$C^\infty$ line bundle over $S$, respectively endowed with a
symplectic form $\omega$ and a hermitian metric $h$. The group of
symmetries relevant for our moment-map construction is the
(Hamiltonian) \emph{extended gauge group} $\cX$ of $(L,h)$ and
$(S,\omega)$, given by an extension
\begin{equation}\label{eq:coupling-term-moment-map-S}
1\to\cG\lto\cX\lra{\pr}\cH\to 1,
\end{equation}
of the group $\cH$ of Hamiltonian symplectomorphisms of $(S,\omega)$
by the unitary gauge group $\cG$ of $(L,h)$. More precisely, $\cX$ is
the group of automorphisms of the hermitian line bundle $(L,h)$ that
cover elements of the group $\cH$, and $\pr$ maps any element of $\cX$
into the element of $\cH$ that it covers.

For each unitary connection $A$ on $(L,h)$, $A\zeta$ denotes the
corresponding vertical component of a vector field $\zeta$ on the
total space of $L$, and $A^\perp y$ denotes the horizontal lift of a
vector field $y$ on $S$ to a vector field on the total space of
$L$. Then the decompositions $\zeta=A\zeta+A^\perp y$, with
$y=\pr(\zeta)$, determine a vector-space splitting of the Lie-algebra
short exact sequence
\begin{equation}
\label{eq:Lie-algebras-ses}
0\to\LieG\lto\LieX\lra{\pr}\LieH\to 0
\end{equation}
associated to~\eqref{eq:coupling-term-moment-map-S}, because
$A^\perp\eta\in\LieX$ for all $\eta\in\LieH$. Note also that the
equation
\begin{equation}\label{eq:hamiltonian-vector-field}
\eta_f\lrcorner\omega=df
\end{equation}
determines an isomorphism between the space $\LieH$ of Hamiltonian
vector fields $\eta=\eta_f$ on $S$, and the space
$C_0^\infty(S,\omega)$ of smooth functions $f$ on $S$ such that
$\int_S f\omega=0$.

We start describing a Hamiltonian $\cX$-action on the space
$\Omega^0(L)$ of smooth global sections of $L$ over $S$. This vector
space has a symplectic form $\omega_\Omega$ determined by $\omega$ and
$h$, given by
\[
\omega_\Omega(\dot\phi_1,\dot\phi_2)=-\operatorname{Im}\int_S(\dot\phi_1,\dot\phi_2)_h\omega,
\]
where $\dot\phi_1,\dot\phi_2\in\Omega^0(L)$ are regarded as tangent
vectors at any $\phi\in\Omega^0(L)$. The 2-form $\omega_\Omega$ is
exact, that is,
\[
\omega_\Omega=d\sigma,
\]
where the 1-form $\sigma$ on $\Omega^0(L)$ is given by 
\[
\sigma_{|\phi}(\dot \phi)=-\operatorname{Im}\int_S(\dot\phi,\phi)_h\omega,
\]
for all $\phi\in\Omega^0(L)$ and $\dot\phi\in\Omega^0(L)\cong
T_\phi\Omega^0(L)$. Furthermore, $\omega_\Omega$ is a K\"ahler 2-form
with respect to the canonical complex structure on $\Omega^0(L)$ given
by multiplication by $i=\sqrt{-1}$.

We observe now that $\cX$ has a canonical action on $\Omega^0(L)$,
defined by
\begin{equation}\label{eq:cX-action}
(g\cdot\phi)(x)\defeq g(\phi(\pr(g)^{-1}x)),
\end{equation}
for all $g\in\cX, \phi\in\Omega^0(L), x\in S$, where $\pr$ is the map
in~\eqref{eq:coupling-term-moment-map-S}. 

\begin{lemma}
\label{lem:mmapc}
The $\cX$-action on $\Omega^0(L)$ is Hamiltonian, with equivariant moment map
\[
\mu\colon\Omega^0(L)\lto(\LieX)^*
\]
given by  
\[
\langle\mu,\zeta\rangle = - \sigma(Y_{\zeta}),
\]
where $Y_\zeta$ denotes the infinitesimal action of $\zeta\in\LieX$ on
$\Omega^0(L)$. For any choice of unitary connection $A$ on $L$, the
moment map is given explicitly by
\begin{equation}\label{eq:mmap1}
\begin{split}
\langle\mu(\phi),\zeta\rangle 
& = \frac{i}{2}\int_S A \zeta |\phi|^2_h\omega - \frac{i}{2}\int_S f d(d_A\phi,\phi)_h
\end{split}
\end{equation}
for all $\phi\in\Omega^0(L)$ and $\zeta \in \LieX$ covering $\eta_f
\in \Lie \cH$, with $f \in C^\infty_0(S)$.
\end{lemma}

\begin{proof}
The first part follows trivially because $\omega_\Omega=d\sigma$ and
$\sigma$ is $\cX$-invariant. To prove~\eqref{eq:mmap1}, we fix a unitary
connection $A$, so the infinitesimal action of $\LieX$ on
$\Omega^0(L)$ is given by \cite{GFR}
\[
Y_{\zeta|\phi} = - \check{\zeta} \lrcorner d_A\phi + A \zeta \cdot \phi,
\]
for all $\zeta\in\LieX$ and $\phi\in\Omega^0(L)$, with
$\check{\zeta}\defeq\pr(\zeta)$. Then $\check{\zeta}=\eta_f$, where
$f\in C^\infty_0(S)$, so
\begin{equation*}
\begin{split}
\langle\mu(\phi),\zeta\rangle & = \frac{i}{2}\int_S (-\check \zeta \lrcorner d_A\phi + A \zeta \cdot \phi,\phi)_h\omega\\
& = \frac{i}{2}\int_S (A \zeta \cdot \phi,\phi)_h\omega - \frac{i}{2}\int_S f d(d_A\phi,\phi)_h,
\end{split}
\end{equation*}
where we have used the identity
\[
(\check{\zeta} \lrcorner d_A \phi) \omega = - df \wedge d_A \phi.
\qedhere
\]
\end{proof}

\subsection{From K\"ahler reduction to gravitating vortices}
\label{sub:Kaher-reduction}

Let $\cJ$ and $\cA$ be the spaces of almost complex structures on $S$
compatible with $\omega$ and unitary connections on $(L,h)$,
respectively; their respective elements will usually be denoted $J$
and $A$. The spaces $\cJ$ and $\cA$ have a natural action by $\cX$ and
admit $\cX$-invariant symplectic structures $\omega_\cJ$ and
$\omega_\cA$ induced by $\omega$. Consider now the space of triples
\begin{equation}\label{eq:spc-triples}
\cJ\times\cA\times\Omega^0(L),
\end{equation}
endowed with the symplectic structure
\begin{equation}\label{eq:symplecticT}
\omega_\cJ + 4 \alpha \omega_\cA + 4\alpha \omega_\Omega
\end{equation}
(for any non-zero coupling constant $\alpha$). By Lemma
\ref{lem:mmapc} combined with \cite[Proposition 2.3.1]{GF}, the diagonal
action of $\cX$ on this space is Hamiltonian, with equivariant moment
map $\mu_\alpha\colon\cJ\times\cA\times\Omega^0(L)\to(\LieX)^*$ given
by
\begin{equation}\label{eq:mutriples}
\begin{split}
\langle \mu_{\alpha}(J,A,\phi),\zeta\rangle =& 4i\alpha\int_S A\zeta \(i\Lambda F_A + \frac{1}{2}|\phi|^2_h - \frac{\tau}{2}\)\omega\\
&- \int_S  f\(S_J + 2i\alpha \(d(d_A\phi,\phi)_h - \tau \Lambda F_A\)\)\omega,
\end{split}
\end{equation}
for any choice of a parameter $\tau \in \RR$.

To make the link with the gravitating vortex equations
\eqref{eq:gravvortexeq1}, consider the space of `integrable triples'
\[
\mathcal{T}\subset\cJ\times\cA\times\Omega^0(L)
\]
defined by 
\begin{equation}\label{eq:cT}
\mathcal{T}\defeq\{(J,A,\phi)\,\mid\,\dbar_A \phi = 0\}.
\end{equation}
Then $\mathcal{T}$ is a
complex submanifold (away from its singularities) for the product
formally integrable almost complex structure on the
space~\eqref{eq:spc-triples} (see~\cite[(2.45)]{AGG}). Moreover, it is
preserved by the $\cX$-action and, by the condition $\alpha > 0$, it inherits a Hamiltonian
action for the K\"ahler form induced by \eqref{eq:symplecticT}.

\begin{proposition}\label{prop:momentmap-inttriples}
  The $\cX$-action on $\mathcal{T}$ is Hamiltonian with
  $\cX$-equivariant moment map $\mu_{\alpha}\colon\cT\to(\LieX)^*$
  given by
\begin{equation}\label{eq:prop-mutriples}
\begin{split}
\langle \mu_{\alpha}(J,A,\phi),\zeta\rangle & = 4i\alpha\int_S A\zeta \(i\Lambda F_A + \frac{1}{2}|\phi|^2_h - \frac{\tau}{2}\)\omega\\
&- \int_M  f\(S_J + \alpha \Delta_\omega |\phi|^2_h - 2\alpha\tau i\Lambda F_A\)\omega.
\end{split}
\end{equation}
for all $(J,A,\phi)\in\cT$ and $\zeta\in\LieX$ covering $\eta_f \in
\LieH$, where $f \in C_0^\infty(S)$.
\end{proposition}

\begin{proof}
Since $(J,A,\phi) \in \cT$, we have $\dbar_A \phi = 0$ and hence
\begin{align*}
\Delta_\omega |\phi|^2_h & = 2i \Lambda \dbar \partial |\phi|^2_h = 2i \Lambda d (\partial_A \phi,\phi)_h = 2i \Lambda d (d_A \phi,\phi)_h.
\end{align*}
The statement follows now from \eqref{eq:mutriples}.
\end{proof}

It can be readily checked that the zeros of the moment map
$\mu_\alpha$, restricted to the space of integrable pairs, correspond
to solutions of the gravitating vortex equations
\begin{equation}\label{eq:gravvortexeq2}
\begin{split}
i\Lambda F_A + \frac{1}{2}(|\phi|_h^2-\tau) & = 0,\\
\dbar_A \phi & = 0,\\
S_J + \alpha(\Delta_\omega + \tau) (|\phi|_h^2 -\tau) & = c,
\end{split}
\end{equation}
where the topological constant $c \in \RR$ is explicitly given by
\begin{equation}\label{eq:constantcbis}
c = \frac{2\pi(\chi(S) - 2\alpha\tau c_1(L))}{\Vol_\omega(S)}.
\end{equation}
Given a solution of \eqref{eq:gravvortexeq2}, considering the complex
structure on $S$ given by $J$, the holomorphic structure on the line
bundle $L$ given by $A$ and the holomorphic section $\phi$, we can
regard $(\omega,h)$ as a solution of the gravitating vortex equations
\eqref{eq:gravvortexeq1} as originally stated in Section
\ref{sec:gravvort}. Conversely, any solution $(\omega,h)$ of
\eqref{eq:gravvortexeq1}, for a holomorphic line bundle with a global
section over a compact Riemann surface, determines a solution of
\eqref{eq:gravvortexeq2} by taking $A$ to be the Chern connection of
$h$.

\subsection{Moduli spaces}
\label{sub:moduli}

We define the \emph{moduli space of gravitating vortices} $\cM_\alpha$
as the space of solutions $(J,A,\phi)$ of the gravitating vortex
equations~\eqref{eq:gravvortexeq2}, modulo the action of the
Hamiltonian extended gauge group $\cX$ defined
in~\eqref{eq:coupling-term-moment-map-S}.
By Proposition \ref{prop:momentmap-inttriples}, this is a symplectic
quotient
\begin{equation}\label{eq:sympquot}
\cM_\alpha=\mu_\alpha^{-1}(0)/\cX,
\end{equation}
so away from its singularities, it is a K\"ahler quotient for the
action of $\cX$ on the smooth part of $\cT$, equipped with a K\"ahler
structure induced by the restriction of~\eqref{eq:symplecticT}, that
may be interpreted as a generalized Weil--Petersson metric (see,
e.g.,~\cite[Theorem~5.7]{Fj} for a similar construction). Note that
the moduli space of solutions of the vortex
equations~\eqref{eq:vortexeq} also has a K\"ahler-reduction
interpretation, and can be identified as a complex manifold with the
$N$-th symmetric product $S^N\Sigma$ of the Riemann surface, where
$N=c_1(L)$ (see~\cite[Theorem~4.7]{Brad},~\cite{G1},
and~\cite[p. 92]{G3}).

Recall now that the link between the Teichm\"uller space and the
moduli space of Riemann surfaces involves the mapping class group
$\Gamma=\pi_0(\Diff_\omega(S))=\Diff_\omega(S)/\Diff_{\omega,0}(S)$,
where $\Diff_\omega(S)$ is the group of symplectomorphisms of
$(S,\omega)$, and $\Diff_{\omega,0}(S)$ is its identity
component. Likewise, to construct moduli spaces parametrizing complex
structures together with effective divisors, we consider a
`bundle-automorphism class group'
\[
\widetilde{\Gamma}\defeq\pi_0(\cX_\omega)=\cX_\omega/\cX_{\omega,0},
\]
where the `large' extended gauge group $\cX_\omega$ of $S$ and $(L,h)$
is the group of automorphisms of the hermitian line bundle $(L,h)$
that cover elements of $\Diff_\omega(S)$, and $\cX_{\omega,0}$ is its
identity component.
It turns out that an element of $\cX_\omega$ covers an element of
$\Diff_{\omega,0}(S)$ if and only if it is in
$\cX_{\omega,0}\subset\cX_\omega$ (see,
e.g.,~\cite[p. 280]{Banyaga1}), that is, we have a group extension
\[
1\to \cG\lto\cX_{\omega,0}\lto\Diff_{\omega,0}(S)\to 1,
\]
so the quotient of $\cX_{\omega,0}$ by the Hamiltonian extended gauge group
$\cX\subset\cX_{\omega,0}$ defined in~\eqref{eq:coupling-term-moment-map-S}
is a $2g$-torus (where $g$ is the genus of $S$), namely
\[
A_S\defeq H^1(S,\RR)/H^1(S,\ZZ)\cong\Diff_{\omega,0}(S)/\cH\cong\cX_{\omega,0}/\cX.
\]
Hence the $\cX_{\omega,0}$-action on the space $\cT$ of triples $(J,A,\phi)$
given by pull-back induces an action of the torus $A_S$ on
$\cM_\alpha$, and we define the \emph{vortex Teichm\"uller space} as
the orbit space
\[
\text{Teich}_\alpha\defeq\cM_\alpha/A_S=\mu_\alpha^{-1}(0)/\cX_{\omega,0}.
\]
%
Finally we define the \emph{vortex Riemann moduli space} as
\[
\chM_\alpha\defeq\mu_\alpha^{-1}(0)/\cX_\omega=\cM_\alpha/(\cX_\omega/\cX)=\text{Teich}_\alpha/\widetilde{\Gamma},
\]
where $\widetilde{\Gamma}$ plays the role of the mapping class group
$\Gamma=\pi_0(\Diff_\omega(S))$ in the standard construction of the
moduli of Riemann surfaces as the orbit space of the Teichm\"uller
space by the $\Gamma$-action. As in the classical connection between
the complex analytic and algebro-geometric descriptions of the moduli
space of Riemann surfaces, we expect that $\chM_\alpha$ is isomorphic
to the complex analytic space underlying the algebro-geometric moduli
space of smooth complex projective curves equipped with an effective
divisors (see, e.g.,~\cite[\S 2.1.3]{Hassett}). In fact,
Theorem~\ref{thm:higher-genus.intro} should be a basic ingredient to
obtain such an isomorphism in genus $g\geq 2$.



\section{A Futaki invariant for gravitating vortices}\label{sec:futaki}

Relying on the moment map picture provided in Section \ref{sec:mmap}, we
use now the general method in~\cite[Section 3]{AGG} to construct a Futaki
invariant for the gravitating vortex equations. As an application, we give an affirmative answer to Yang's Conjecture \ref{conj:Yangintro}, by establishing the unstable case of Theorem \ref{th:Yangconjectureintro} for divisors supported at one or two points.

\subsection{Automorphism groups}\label{sec:Aut}

Let $\Sigma$ be a compact Riemann surface, of arbitrary genus $g(\Sigma)$, $L$ a holomorphic line bundle over $\Sigma$ with $c_1(L) >0$, and $\phi\in H^0(\Sigma,L)$ a non-zero section. The Futaki invariant for the gravitating vortex
equations~\eqref{eq:gravvortexeq2} is a character of the Lie algebra of infinitesimal automorphisms of $(\Sigma,L,\phi)$. Before we introduce our invariant in Section \ref{sec:FutakiVortex}, in this section we study this Lie algebra in some detail. For some calculations, it will handy to have a description of the corresponding automorphism group.

The \emph{automorphism group} of the pair $(\Sigma,L)$
is the group $\Aut(\Sigma,L)$ of $\CC^*$-equivariant automorphisms of
the total space of the holomorphic line bundle $L$, with the
$\CC^*$-action on $L$ given by fibrewise scalar multiplication. As
$\CC^*$-equivariant automorphisms of $L$ preserve the zero section, there is a canonical exact sequence
\begin{equation}\label{eq:seqautgeneral}
1\to\CC^*\lto\Aut(\Sigma,L)\lra{\pr}\Aut(\Sigma), 
\end{equation}
where $\Aut(\Sigma)$ is the automorphism group of $\Sigma$ and the
right-hand arrow maps any $g\in\Aut(\Sigma,L)$ into the unique
$\pr(g)=\check{g}\in\Aut(\Sigma)$ covered by $g$. 
The \emph{automorphism group} of the triple $(\Sigma,L,\phi)$
is the isotropy subgroup
\begin{equation}\label{eq:Aut-triple}
\Aut(\Sigma,L,\phi)\subset\Aut(\Sigma,L)
\end{equation}
of $\phi$ for the induced action of $\Aut(\Sigma,L)$ on
$H^0(\Sigma,L)$. 
By a result of Morimoto~\cite[p. 158]{Morimoto},
$\Aut(\Sigma,L)$ (with the compact-open topology) is a complex Lie
group, and the action of $\Aut(\Sigma,L)$ on $L$ and the right-hand
map in~\eqref{eq:seqautgeneral} are both holomorphic~\cite[Section
7]{Morimoto}. Let
\begin{equation}\label{eq:LieAut-triple}
\Lie\Aut(\Sigma,L,\phi)\subset\Lie\Aut(\Sigma,L)
\end{equation}
be the Lie algebras of $\Aut(\Sigma,L,\phi)\subset\Aut(\Sigma,L)$,
respectively. By definition of $\Aut(\Sigma,L)$, 
$\Lie\Aut(\Sigma,L)$ consists of the $\CC^*$-invariant holomorphic vector fields on the
total space of $L$, and so $\Lie\Aut(\Sigma,L,\phi)$ consists of those
vector fields with zero infinitesimal action on $\phi\in H^0(\Sigma,L)$.

We will use the following description of $\Lie\Aut(\Sigma,L)$ (see
e.g.~\cite[p. 490]{D7}).

\begin{lemma}\label{lem:yholomorphic}
A $\CC^*$-invariant vector field $y$ on the total space of $L$ belongs to $\Lie \Aut(\Sigma,L)$ if and only if for any choice of hermitian metric $h$ on $L$ the following equation is satisfied:
\begin{equation}\label{eq:yholomorphic}
\dbar (A_h y) + \iota_{\check y^{1,0}}F_h = 0.
\end{equation}
\end{lemma}

In \eqref{eq:yholomorphic} $A_h$ is the Chern connection of the hermitian
metric $h$ and $\check y$ denotes the unique holomorphic vector field on $\Sigma$ covered by $y$. 

Our partial characterization of the complex Lie group $\Aut(\Sigma,L,\phi)$ breaks up into two separate cases, depending on whether $\Sigma$ has positive or zero genus, respectively. The positive genus case is a formality. 
We give a proof for the benefit of the reader. 

\begin{proposition}\label{prop:automorphism-groups.positive-genus}
If 
$g(\Sigma)>0$, then the group $\Aut(\Sigma,L,\phi)$ is discrete. 
\end{proposition}

\begin{proof}
Let $y\in\Lie\Aut(\Sigma,L,\phi)$. We will show that $y$ is vertical, that is, the holomorphic vector field
$\check{y}\in\Lie\Aut(\Sigma)$ covered by $y$ is zero. The result will
then follow because the only holomorphic vertical vector fields are
$y=t\mathbf{1}$, for constant $t\in\CC$, where $\mathbf{1}$ is the
tautological vector field on the fibres, and so the condition that $y$
fixes $\phi\neq 0$ implies $y=0$. If $g(\Sigma)>1$, the fact that
$\check{y}=0$ is deduced, e.g., because there exists a negative
curvature K\"ahler metric on $\Sigma$. 
Suppose now $g(\Sigma)=1$. Then, there exists a flat K\"ahler metric on
$\Sigma$, so either $\check{y}$ has no zeros or it vanishes
identically, as it is necessarily parallel with respect to the flat
K\"ahler metric (see e.g.~\cite{Gauduchon}). Now, by assumption $c_1(L)>0$, so $L$
is ample~\cite[Ch IV, Cor. 3.3]{Ha},
and therefore there exists a hermitian metric $h$ on $L$ such that
$\omega=iF_h$ is a K\"ahler metric on
$\Sigma$. But Lemma~\ref{lem:yholomorphic} applied to $y$ and the
hermitian metric $h$ imply
\begin{equation}\label{eq:holy}
- i \dbar(A_h y) = i_{\check{y}^{1,0}}\omega, 
\end{equation}
that is, $A_h y$, identified with a complex function on $\Sigma$, is a
complex potential for $\check y$. Therefore, $\check y$ vanishes
somewhere on $\Sigma$ (see \cite{LS1}) and hence it identically
vanishes.
\end{proof}

We now turn to the more interesting genus zero case $\Sigma = \mathbb{P}^1$. We make the identification $L=\cO_{\PP^1}(N)$, with $N\defeq
c_1(L)>0$, and fix homogeneous coordinates $[x_0,x_1]$ on
$\PP^1$. Then $H^0(\mathbb{P}^1,L)\cong S^N(\CC^2)^*$ is the space of degree
$N$ homogeneous polynomials in the coordinates $x_0,x_1$, so it is a
$\operatorname{GL}(2,\CC)$-representation, where $g\in\GL(2,\CC)$ maps
a polynomial $p(x_0,x_1)$ into the polynomial
$p(g^{-1}(x_0,x_1))$. Furthermore,
$\Aut(\PP^1)=\operatorname{PGL}(2,\CC)$ and the sequence
\eqref{eq:seqautgeneral} is a short exact sequence
\[
1\to\CC^*\lto\Aut(\PP^1,L)\lra{\pr}\operatorname{PGL}(2,\CC)\to 1.
\]
%
%
Here, the third arrow is surjective, since it is the horizontal arrow
in a commutative diagram
\begin{equation}\label{eq:SQDiagram2}
  \xymatrix{
  	\operatorname{GL}(2,\CC) \ar@{->>}[d]_-{\rho} \ar@{->>}[dr] &   \\
    \Aut(\PP^1,L) \ar[r]^-{\pr} & \operatorname{PGL}(2,\CC),
  }
\end{equation}
where the diagonal arrow is the canonical surjective morphism, and the
vertical arrow $\rho$ is the canonical $\GL(2,\CC)$-linearization of
$L$, that is, it is the surjective morphism induced by the $\operatorname{GL}(2,\CC)$-representation $H^0(\PP^1,L)$. 
Note that an element in the centre,
$\lambda\in\CC^*\subset\operatorname{GL}(2,\CC)$, acts via $\rho$ on
$L$ by fibrewise multiplication by $\lambda^{-N}$.

Let $\phi \in H^0(\mathbb{P}^1,L)\cong S^N(\CC^2)^*$. It is well-known that an element of $\operatorname{PGL}(2,\CC)$ which fixes more than two points on $\PP^1$ must be the identity. Therefore, if $\phi$ has more than two zeros, $\Aut(\PP^1,L,\phi)$ is necessarily trivial. Thus, by a suitable choice of homogeneous coordinates we can assume \begin{equation}\label{eq:phi-Einstein-Bogomonyi}
\phi\cong x_0^{N- \ell}x_1^\ell,
\end{equation}
with $0\leq\ell<N$ (the case $\ell=0$ corresponds to a Higgs field
$\phi$ that has only one zero). 

\begin{lemma}\label{lem:Yangconjextended}
Let $\phi \in  H^0(\mathbb{P}^1,L)$ as in \eqref{eq:phi-Einstein-Bogomonyi}.

\begin{itemize}

\item If $\ell = 0$, then $\Aut(\PP^1,L,\phi)\cong\CC^*\rtimes\CC$. More explicitly, $\Aut(\PP^1,L,\phi)$ is the image under $\rho$ in \eqref{eq:SQDiagram2} of the subgroup
\begin{equation}\label{eq:automorphism-Yang-conj}
\( \begin{array}{cc}
1 & 0 \\
* & *
\end{array} \) \subset \operatorname{GL}(2,\CC).
\end{equation}

\item If $\ell \neq 0$, then $\Aut(\PP^1,L,\phi)$ is
given by the image of the standard maximal torus $\CC^* \times \CC^*
\subset \operatorname{GL}(2,\CC)$ under the morphism
$$
\rho_\ell \colon \CC^* \times \CC^* \to \Aut(\PP^1,L)
$$
defined by
$$
\rho_\ell(\lambda_0,\lambda_1) = \lambda_0^{N- \ell} \lambda_1^{\ell}\rho(\lambda_0,\lambda_1)
$$
where $\lambda_0^{N- \ell} \lambda_1^{\ell}$ acts on $L$ by multiplication on the fibres.
\end{itemize}
\end{lemma}

The proof follows easily from the surjectivity of $\rho$ in \eqref{eq:SQDiagram2}.

\subsection{Definition of the Futaki invariant}\label{sec:FutakiVortex}

By Proposition \ref{prop:automorphism-groups.positive-genus}, $\Lie\Aut(\Sigma,L,\phi)$ is trivial if $g(\Sigma) > 0$, and therefore we assume $\Sigma = \mathbb{P}^1$ throughout this section. 

Fix $\alpha,\tau$, and $\Vol(\PP^1)$ positive real numbers. We denote by $B$ the space of pairs $(\omega,h)$ consisting of a K\"ahler form $\omega$ on $\PP^1$ with volume $\Vol(\PP^1)$, and a hermitian metric $h$ on $L$. Throughout Section \ref{sec:futaki}, we will view the gravitating
vortex equation~\eqref{eq:gravvortexeq2} as equations where the
unknowns belong to the space $B$.
Define a map 
\begin{equation}\label{eq:futakigravvort.arrow}
\cF_{\alpha,\tau}\colon\Lie\Aut(\PP^1,L,\phi)\lto\CC,
\end{equation}
by the following formula, for all $y\in\Lie\Aut(\PP^1,L,\phi)$, where
$(\omega,h)\in B$:
\begin{equation}\label{eq:futakigravvort}
\begin{split}
\langle\cF_{\alpha,\tau},y\rangle & = 4i\alpha\int_{\PP^1} A_h y \(i\Lambda_\omega F_h + \frac{1}{2}|\phi|^2_h - \frac{\tau}{2}\)\omega
- \int_{\PP^1}  \varphi \(S_\omega + \alpha \Delta_\omega |\phi|_h^2 - 2i \alpha\tau \Lambda_\omega F_h\)\omega.
\end{split}
\end{equation}
Here, $A_h$ is the Chern connection of $h$ on $L$, $A_hy\in
C^\infty(\PP^1,i\RR)$ is the vertical projection
of $y$ with respect to $A_h$, and the complex valued function
$\varphi$ on $\PP^1$ is defined as follows. Let $\check{y}$ be the
holomorphic vector field on $\PP^1$ covered by $y$ and $A_h^\perp\check{y}$ its
horizontal lift to a vector field on the total space of $L$ given by
the connection $A_h$, so $y$ has a decomposition
\begin{equation}\label{eq:holvectfield-Vert+Horiz.2}
y = A_h y+A_h^\perp\check{y}  
\end{equation}
into its vertical and horizontal components. Then
$\varphi\defeq\varphi_1+i\varphi_2\in C^\infty(\PP^1,\CC)$ is
determined by the unique decomposition
\begin{equation}\label{eq:vect-field-decomposition.2}
\check{y}=\eta_{\varphi_1}+J\eta_{\varphi_2} 
\end{equation}
associated to the K\"ahler form $\omega$ (see~\cite{LS1}), where
$\eta_{\varphi_j}$ is the Hamiltonian vector field of the function
$\varphi_j\in C^\infty_0(\PP^1,\omega)$ on $(\PP^1,\omega)$
(see~\eqref{eq:hamiltonian-vector-field}), for $j=1,2$, and $J$ is the
almost complex structure of $\PP^1$. 
Note that the previous decomposition uses the fact that $\PP^1$ is simply connected.

The non-vanishing of $\cF_{\alpha,\tau}$ is our first obstruction to
the existence of gravitating vortices.

\begin{proposition}\label{prop:futakibis}
The map $\cF_{\alpha,\tau}$ is independent of the choice of 
$(\omega,h)\in B$. It is a character of the Lie algebra
$\Lie\Aut(\PP^1,L,\phi)$, that vanishes identically if there exists a
solution $(\omega,h)$ of the gravitating vortex equations
\eqref{eq:gravvortexeq2} on $(\PP^1,L,\phi)$, with volume
$\Vol(\PP^1)$.
\end{proposition}

By analogy with Futaki's obstruction to the existence of K\"ahler--Einstein 
metrics~\cite{Futaki1} $\cF_{\alpha,\tau}$ will be called the \emph{Futaki invariant} for the
gravitating vortex equations~\eqref{eq:gravvortexeq2}, with symmetry
breaking parameter $\tau$, coupling constant $\alpha$, and volume
$\Vol(\Sigma)$. Note that the term $\int_{\PP^1}\varphi S_\omega\omega$
in~\eqref{eq:futakigravvort} is in fact the original Futaki
character.

\begin{proof}
In the framework of Section \ref{sec:mmap}, we consider the
$C^\infty$ manifold $S^2$ underlying the Riemann sphere $\PP^1$. For $b = (\omega,h) \in B$, let $\cT_b$ be the associated space of `integrable triples'
\[
\cT_b \subset\cJ_\omega \times\cA_h \times\Omega^0(L)
\]
defined in~\eqref{eq:cT}, with a distinguised point $t_b = (J,A,\phi)$ given by the triple $(\PP^1,L,\phi)$. Recall that $\cT_b$ is endowed with a
(formally) integrable almost complex structure $\mathbf{I}$,  and a K\"ahler metric (as $\alpha>0$), with  compatible symplectic structure $\omega_\alpha$ as in \eqref{eq:symplecticT}. Furthermore, there is a Hamiltonian action of the extended gauge group $\cX_b$ on $\cT_b$ such that if $b = (\omega,h)$ is a solution of the gravitating vortex equations, then the triple $t_b = (J,A,\phi)$ is a zero of a moment map ~\eqref{eq:prop-mutriples}. Then, we can construct a $\CC$-linear map
\[
\cF_b\colon\Lie\Aut(\PP^1,L,\phi)\lto\CC
\]
as in~\cite[(3.108)]{AGG}. The explicit formula for this map is
obtained as in~\cite[(3.126)]{AGG}, replacing the moment map
formula~\cite[(2.44)]{AGG} by~\eqref{eq:prop-mutriples}. The proof now follows as for~\cite[Theorem~3.9]{AGG}.

\end{proof}

\subsection{An application of the Futaki character}\label{sec:futakievaluation}

The following result illustrates the non-vanishing of the Futaki
character as an obstruction to the existence of gravitating
vortices.

\begin{theorem}\label{th:Yangconjecture.2}
Assume $\alpha > 0$. Then, there is no solution of the gravitating vortex equations for $(\PP^1,L,\phi)$ with $\phi$ vanishing exactly at one point, or at two points with different multiplicities.
\end{theorem}

The proof of Theorem~\ref{th:Yangconjecture.2} follows from
Proposition~\ref{prop:futakibis} and a direct calculation of the
Futaki invariant on a holomorphic line bundle $L=\cO_{\PP^1}(N)$ over
$\PP^1$. In order to show this, we fix homogeneous coordinates $[x_0,x_1]$ and follow
the notation of Section \ref{sec:Aut}. 
We wish to evaluate the
Futaki invariant for $(\PP^1,L,\phi)$, when $\phi$ is as in \eqref{eq:phi-Einstein-Bogomonyi}. By Lemma \ref{lem:Yangconjextended}, the Lie algebra element
\begin{equation}\label{eq:y}
y = \( \begin{array}{cc}
0 & 0 \\
0 & 1
\end{array} \) \in \mathfrak{gl}(2,\CC)
\end{equation}
can be identified with an element in $\Lie\Aut(\PP^1,L,\phi)$ for any
$0\leq\ell<N$.

\begin{lemma}\label{lem:evaluationfutaki}
\begin{equation}
\langle \cF_{\alpha,\tau},y\rangle = 2\pi i\alpha(2N-\tau)(2\ell - N)
\end{equation}
\end{lemma}

\begin{proof}
Without loss of generality, we assume $\Vol(\PP^1) = 2\pi$ in the definition of the Futaki invariant. We will apply formula \eqref{eq:futakigravvort} to the
pair $(\omega_{FS},h_{FS}^N)$ consisting of the Fubini--Study metric
$\omega_{FS}$ on $\PP^1$, normalized so that $\int_{\PP^1}\omega_{FS}
= 2\pi$, and the Fubini--Study metric $h_{FS}^N$ on $L =
\mathcal{O}_{\PP^1}(N)$.  We choose coordinates $z = \frac{x_1}{x_0}$,
so that the vector field on $\PP^1$ induced by $y$ and the holomorphic
section $\phi$ are
$$
\check y^{1,0} = z \frac{\partial}{\partial z}, \qquad \phi = z^\ell.
$$
In these coordinates, we also have
$$
\omega_{FS} = \frac{i dz \wedge d\overline{z}}{(1 + |z|^2)^2}, \qquad h_{FS}^N = \frac{1}{(1 + |z|^2)^{N}}, 
$$
and $y = J \eta_{\varphi_2}$, with global complex potential $\varphi = i \varphi_2$ given by
$$
\varphi = \frac{i}{2} \frac{|z|^2 - 1}{|z|^2 + 1}.
$$
Hence, by Lemma \ref{lem:Yangconjextended}, 
the infinitesimal action of $y$ induces a vector field on the total
space of $L$, also denoted $y$, with vertical part 
\begin{align*}
A_{h_{FS}^N} y & = \ell + \iota_{\check y^{1,0}} \partial \log h_{FS}^N\\
& = \ell - N \frac{|z|^2}{1 + |z|^2}.
\end{align*}
Applying these formulae in~\eqref{eq:futakigravvort}, we obtain
\begin{align*}
\langle \cF_{\alpha,\tau},y\rangle & = 4i\alpha\int_{\PP^1} A_{h_{FS}^N}y \(N + \frac{1}{2}|\phi|^2_{h^N_{FS}} - \frac{\tau}{2}\)\omega_{FS} - \alpha \int_{\PP^1}  \varphi \Delta_{\omega_{FS}} |\phi|_{h_{FS}}^2 \omega_{FS}\\
& = 2i\alpha(2N - \tau)\int_{\PP^1} (A_{h_{FS}^N}y) \omega_{FS} + 2\alpha\int_{\PP^1} (iA_{h_{FS}^N}y - 2\varphi)|\phi|^2_{h^N_{FS}} \omega_{FS}
\end{align*}
where we have used the facts that
$i\Lambda_{\omega_{FS}}F_{h_{FS}^N}=N$, $S_{\omega_{FS}}$ is constant,
and $\varphi$ is normalised so that
\[
\Delta_{\omega_{FS}}\varphi = 4 \varphi,
\]
so in particular $\int_{\PP^1} \varphi \omega_{FS} = 0$. Using now the explicit formula
\[
|\phi|_{h_{FS}^N}^2 = \frac{|z|^{2\ell}}{(1 + |z|^2)^{N}},
\]
we have
\begin{align*}
\int_{\PP^1} (A_{h_{FS}^N}y) \omega_{FS} & = \int_{\CC} \(\ell - N \frac{|z|^2}{1 + |z|^2}\)\frac{1}{(1 + |z|^2)^2}idz \wedge d\overline{z}\\
& = 4\pi \int_0^\infty \(\ell - N \frac{r^2}{1 + r^2}\)\frac{r}{(1 + r^2)^2}dr\\
& = 4\pi \Bigg{[} \frac{2Nr^2 + N - 2(r^2+1)\ell}{4(r^2+1)^2}\Bigg{]}^\infty_0\\
& = \pi (2\ell - N)
\end{align*}
and also, using that $\ell < N$,
\begin{align*}
\int_{\PP^1} (iA_{h_{FS}^N}y - 2\varphi)|\phi|^2_{h^N_{FS}} \omega_{FS} & = i\int_{\CC} \(\ell - N \frac{|z|^2}{1 + |z|^2} +  \frac{1-|z|^2}{1 + |z|^2}\)\frac{|z|^{2\ell}}{(1 + |z|^2)^{N+2}}idz \wedge d\overline{z}\\
& = 4\pi i \int_0^\infty \(\ell - N \frac{r^2}{1 + r^2} +  \frac{1-r^2}{1 + r^2}\)\frac{r^{2\ell+1}}{(1 + r^2)^{N+2}}dr\\
& = 2\pi i \Bigg{[} \frac{r^{2\ell + 2}}{(1+r^2)^{2N+2}}\Bigg{]}^\infty_0\\
& = 0,
\end{align*}
which completes the proof.
\end{proof}

\begin{proof}[Proof of Theorem~\ref{th:Yangconjecture.2}]
This is now a direct consequence of Proposition \ref{prop:futakibis}
and Lemma \ref{lem:evaluationfutaki}: if there exists a solution of
the gravitating vortex equations for $(\PP^1,L,\phi)$, then
$\cF_{\alpha,\tau} = 0$ and therefore $2\ell = N$ or $\tau = 2N$. The
second case is excluded by Theorem \ref{th:B-GP}.
\end{proof}

Since the Einstein--Bogomol'nyi equations~\eqref{eq:cosmicstrings} are
a particular case of the gravitating vortex
equations~\eqref{eq:gravvortexeq1} on $\PP^1$,
Theorem~\ref{th:Yangconjecture.2} settles Yang's
Conjecture~\ref{conj:Yangintro}.

\begin{corollary}[{Yang's conjecture}]
\label{cor:YangConjecture}
There is no solution of the Einstein--Bogomol'nyi equations for
$N$ strings superimposed at a single point, that is, when $(L,\phi)$
corresponds to a divisor $D=Np$.
\end{corollary}

\begin{remark}\label{rem:second_proof-Yangs_conjecture}
If $\phi$ has only one zero, so that $\ell=0$, by Lemma \ref{lem:Yangconjextended} we could have chosen another Lie
algebra element
\[
y' = \Big{(} \begin{array}{cc}
0 & 0 \\
1 & 0
\end{array} \Big{)}.
\]
However, for this choice $\langle\cF_{\alpha,\tau},y'\rangle=0$, since $[y,y']=y'$ and
$\cF_{\alpha,\tau}$ is a character by Proposition \ref{prop:futakibis}.
\end{remark}

\subsection{Relation with extremal pairs}

In the case $N=1$ and $\ell=0$, there is a simpler proof of the
non-vanishing of the Futaki invariant, which is related to a suitable
notion of extremal pair (cf.~\cite[Definition
4.1]{AGG}). 
Let $\omega$ be a K\"ahler form on $\PP^1$ and $h$ a hermitian metric
on $L$. Associated with the pair $(\omega,h)$ and a constant
$a\in\RR$, we consider a vector field
\[
\zeta_{a,\alpha,\tau}(\omega,h)\defeq ai(i\Lambda F_A + \frac{1}{2}|\phi|^2_h - \frac{\tau}{2})\mathbf{1} + A_h^\perp \eta_{\alpha,\tau}
\]
on the total space of $L$, where $\eta_{\alpha,\tau}$ is the Hamiltonian
vector field of the smooth function
\[
S_\omega+\alpha \Delta_\omega|\phi|^2_h-2\alpha\tau i\Lambda_\omega F_h.
\]
Note that the vector field $\zeta_{a,\alpha,\tau}(\omega,h)$ is
$\CC^*$-invariant (actually it belongs to the extended gauge group
determined by $(\omega,h)$).

\begin{definition}\label{def:extremal_pair}
The pair $(\omega,h)$ is \emph{extremal} if there exists
$a\in\RR_{>0}$ such that
\[
\zeta_{a,\alpha,\tau}(\omega,h)\in\Lie\Aut(\PP^1,L,\phi),
\]
that is, the vector field $\zeta_{a,\alpha,\tau}(\omega,h)$ is
holomorphic and preserves $\phi$.
\end{definition}

The existence of a non-trivial extremal pair $(\omega,h)$ with a fixed
volume $\Vol(\PP^1)$ is an obstruction to the existence of solutions
of the gravitating vortex equations with the same volume, where
non-triviality means $\zeta_{a,\alpha\,\tau}(\omega,h)\neq 0$ for some
$a\in\RR_{> 0}$. This follows from Proposition~\ref{prop:futakibis},
because $\zeta_{a,\alpha\,\tau}(\omega,h)\neq 0$ implies
\[
\langle\cF_{\alpha,\tau},\zeta_{a,\alpha,\tau}(\omega,h)\rangle<0,
\]
as can be shown by applying formula~\eqref{eq:futakigravvort}
to $y=\zeta_{a,\alpha,\tau}(\omega,h)$, using $(\omega,h)$ (cf.~\cite[Proposition 4.2]{AGG}).

\begin{proposition}
The pair $(\omega_{FS},h_{FS})$ is an extremal pair for
$(\PP^1,\mathcal{O}_{\PP^1}(1),\phi)$, with $\phi = x_0$.
\end{proposition}
\begin{proof}
We note that
$$
\Delta_{\omega_{FS}} |\phi|_{h_{FS}}^2 = 2\frac{1- |z|^2}{1 + |z|^2},
$$
which is the Hamiltonian for the vector field
$$
v = -4iy_1 = 4iz\frac{\partial}{\partial z} \in \Lie \Aut(\PP^1,\mathcal{O}_{\PP^1}(1),\phi).
$$
Furthermore, we have
$$
\dbar |\phi|^2_{h_{FS}} = \frac{1}{4}i_{v^{1,0}}\omega_{FS},
$$
and hence the result holds for the choice $a = 8$, by Lemma
\ref{lem:yholomorphic}.
\end{proof}

\begin{remark}\label{rem:extremal_pair}
One can compare the definition of extremal pair for the K\"ahler--Yang--Mills equations in
\cite[Definition 4.1]{AGG} with Definition~\ref{def:extremal_pair} via the
process of dimensional reduction described in~\cite[\S
3.2]{AGG2}. Under this comparison, the former definition corresponds
to the latter only for $a=4$, but clearly the notion of extremal pair for the K\"ahler--Yang--Mills equations can be generalized to
arbitrary $a\in\RR_{>0}$ by considering a modification of the vector
field $\zeta_\alpha$ (see~\cite[(4.136)]{AGG}), with the
Hermite--Yang--Mills term multiplied by $a$, as in
Definition~\ref{def:extremal_pair}.
\end{remark}

\section{From gravitating vortices in $g=0$ to polystability}\label{sec:geodesics}

In this section, we introduce a notion of geodesic stability for the
gravitating vortex equations, which is valid for a surfaces of
arbitrary genus $g$ (Section~\ref{sec:geodesicstability}), prove our
main Theorem~\ref{th:Yangconjectureintro} in $g=0$
(Section~\ref{subsec:geodesics}), and discuss the role of the
Homogeneous Complex Monge--Amp\`ere equation in the problem of
existence and uniqueness of gravitating vortices
(Section~\ref{sec:conjecturemoduli}). Key tools are our description of
the Lie algebra of automorphisms of a triple $(\Sigma,L,\phi)$
carrying gravitating vortices (Section~\ref{sec:Aut})
and the Futaki invariant for the gravitating vortex equations
(Section~\ref{sec:FutakiVortex}).

\subsection{Geodesic stability}\label{sec:geodesicstability}

Let $\Sigma$ be a compact connected Riemann surface of arbitrary
genus, $L$ a holomorphic line bundle over $\Sigma$, $\phi$ a global
holomorphic section of $L$, and $\tau,\alpha\in\RR$, where
$\alpha>0$. Fix $\Vol(\Sigma)>0$.
In this section we construct an obstruction to the existence of
solutions of the gravitating vortex equations that is intimately
related to the geometry of the infinite-dimensional space $B$
consisting of pairs $(\omega,h)$, where $\omega$ is a K\"ahler form on
$\Sigma$ with volume $\Vol(\Sigma)$ and $h$ is a hermitian metric on
$L$.
This space has a structure of symmetric space~\cite[Theorem~3.6]{AGG},
that is, it has a torsion-free affine connection $\nabla$, with
holonomy group contained in the extended gauge group (each point of
$B$ determines one such group) and covariantly constant
curvature. 
The partial differential equations that define the geodesics
$(\omega_t,h_t)$ on $B$, with respect to the connection $\nabla$,
are~\cite[Proposition~3.17]{AGG}
\begin{equation}
\label{eq:geodesicequation2}
\begin{split}
dd^c(\ddot \varphi_t - (d\dot\varphi_t,d\dot\varphi_t)_{\omega_t}) &= 0,\\
\ddot h_t - 2J\eta_{\dot \varphi_t} \lrcorner d\dot h_t + iF_{h_t}(\eta_{\dot \varphi_t},J \eta_{\dot\varphi_t}) &= 0.
\end{split}
\end{equation}
Here, $\omega_t=\omega+dd^c\varphi_t$ with $\varphi_t\in
C^\infty(\Sigma)$, and $\eta_{\dot \varphi_t}$ is the Hamiltonian
vector field of $\dot \varphi_t$ with respect to $\omega_t$, that is,
given by
\[
d \dot \varphi_t = \eta_{\dot \varphi_t} \lrcorner \omega_t.
\]
The long-time existence of smooth geodesics on $B$ \textemdash{} a
very hard analytical open problem \textemdash{} has strong
consequences for the problem of the gravitating vortex equations
\eqref{eq:gravvortexeq2}. Assuming existence of smooth geodesic rays,
that is, smooth solutions $(\omega_t,h_t)$
of~\eqref{eq:geodesicequation2} defined on an infinite interval $0\leq
t<\infty$, with prescribed boundary condition at $t=0$, one can define
a stability condition for the triple $(\Sigma,L,\phi)$. Define a
$1$-form $\sigma_{\alpha,\tau}$ on $B$ by
\begin{align}
\notag
\sigma_{\alpha,\tau}(\dot \omega,\dot h) =
& - 4 \alpha_1 \int_\Sigma h^{-1}\dot{h}\Big{(}i\Lambda_{\omega} F_{h} + \frac{1}{2}|\phi|^2_h - \frac{\tau}{2}\Big{)} \omega\\
\label{eq:sigma-explicit2}
& - \int_\Sigma\dot{\varphi} \(S_{\omega} + \alpha \Delta_\omega |\phi|^2_h - 2\alpha\tau i\Lambda_\omega F_h\) \omega,
\end{align}
where $(\dot \omega, \dot h)$ is a tangent vector to $B$ at
$(\omega,h)$, so $\dot\omega = dd^c\dot \varphi$ with
$\dot{\varphi}\in C_{0}^\infty(\Sigma,\omega)$, that is,
$\dot{\varphi}$ is normalised by the condition
$\int_\Sigma\dot{\varphi}\omega=0$. Then $\sigma_{\alpha,\tau}$
vanishes precisely at the pairs $(\omega,h) \in B$ that are solutions
of the gravitating vortex equations.
As in \cite[Proposition~3.8]{AGG}, 
\begin{equation}\label{eq:ddtsigma}
\frac{d}{dt}\sigma_{\alpha,\tau}(\dot \omega_t,\dot h_t) \geq 0
\end{equation}
along a geodesic ray (cf.~\cite[Proposition 3.10]{AGG}), with speed
controlled by the infinitesimal action of the extended gauge group on
the space $\cT$ in~\eqref{eq:cT} (cf. the proof of \cite[Proposition
3.14]{AGG}), and hence it makes sense to evaluate the \emph{maximal
  weight}
\begin{equation}\label{eq:weight}
w(\Sigma,L,\phi)\defeq\lim_{t \to + \infty}\sigma_{\alpha,\tau}(\dot \omega_t,\dot h_t).
\end{equation}

\begin{definition}[cf. {\cite[Definition 3.13]{AGG}}]\label{def:geodstab}
The triple $(\Sigma,L,\phi)$ is \emph{geodesically semi-stable} if 
\[
w(\Sigma,L,\phi)\geq 0
\]
for every smooth geodesic ray $(\omega_t,h_t)$ on $B$. It is
\emph{geodesically stable} if this inequality is strict whenever
$(\omega_t,h_t)$ is non-constant.
\end{definition}

Under the assumption that $B$ is geodesically convex, that is, any two
points in $B$ can be joined by a smooth geodesic segment, geodesic
semi-stability provides an obstruction to the existence of solutions
of the gravitating vortex
equations~\eqref{eq:gravvortexeq1}. Furthermore, this assumption has
strong consequences for the uniqueness of solutions (cf.
Section~\ref{sec:conjecturemoduli}). The next proposition follows from
the fact that the quantity $\sigma_{\alpha,\tau}(\dot b_t)$ is
increasing along geodesics in $B$ (see~\eqref{eq:ddtsigma}), and
$\sigma_{\alpha,\tau}$ vanishes at the solutions $(\omega,h)\in B$ of
the gravitating vortex equations.

\begin{proposition}[cf. {\cite[Corollary 3.11]{AGG}}]
\label{prop:geod_convex-uniqueness}
Assume that $B$ is geodesically convex. If there exists a solution of the gravitating vortex equations in $B$, then $(\Sigma,L,\phi)$ is geodesically semi-stable. Furthermore, such a solution is unique modulo the action of $\Aut(\Sigma,L,\phi)$.
\end{proposition}

\subsection{The converse of Yang's Existence Theorem}\label{subsec:geodesics}

We are now in a position to prove
Theorem~\ref{th:Yangconjectureintro}. We start with the observation
that the geodesic equation \eqref{eq:geodesicequation2} is independent
of the global section $\phi$ (it is the geodesic equation already
considered in the K\"ahler--Yang--Mills problem~\cite{AGG}), so one
obtains a wealth of geodesic rays starting at any point of $B$.

\begin{lemma}
Let $(\omega,h)\in B$. Then, any $\zeta\in\Lie\Aut(\Sigma,L)$
determines a smooth geodesic ray in $B$ starting at
$(\omega,h)$, given by
\begin{equation}\label{eq:geodesicray}
b_t=(\omega_t,h_t)=(g_t^*\omega,g_t^*h),
\end{equation}
where $g_t\in\Aut(\Sigma,L)$ is the flow of $\zeta$, with initial condition $g_0 = \Id$.
\end{lemma}

We now restrict ourselves to the case $\Sigma=\PP^1$, and evaluate the
maximal weight~\eqref{eq:weight} on a geodesic ray of the
form~\eqref{eq:geodesicray}. Since $\Sigma=\PP^1$ and $L$ are fixed,
throughout Section~\ref{subsec:geodesics} we will denote by
$\cF_{\alpha,\tau}(\phi)$ the Futaki invariant defined in
Section~\ref{sec:futaki}, corresponding to a triple $(\PP^1,L,\phi)$.

\begin{lemma}
\label{lem:weight-P1}
Let $(\omega,h)\in B$ and $\zeta\in\Lie\Aut(\Sigma,L)$. Assume that
the limit
\[
\phi_0\defeq\lim_{t\to+\infty}g_t\cdot{\phi}
\]
exists, where $g_t\in\Aut(\Sigma,L)$ is the flow of $\zeta$, with
initial condition $g_0 = \Id$. Then the maximal weight of
$(\PP^1,L,\phi)$, evaluated at the geodesic ray~\eqref{eq:geodesicray}
starting at $(\omega,h)$, is
\begin{equation}\label{eq:weightFutaki}
w(\PP^1,L,\phi)=\operatorname{Im}\,\langle\cF_{\alpha,\tau}(\phi_0),\zeta\rangle.
\end{equation}
\end{lemma}

Note that the right-hand side of~\eqref{eq:weightFutaki} is well
defined, that is, $\zeta\in\Lie\Aut(\PP^1,L,\phi_0)$, because by the
hypothesis of the lemma, $\phi_0$ is fixed by the one-parameter
subgroup induced by $\zeta$.

\begin{proof}
By \eqref{eq:holvectfield-Vert+Horiz.2} and \eqref{eq:vect-field-decomposition.2}, any $\zeta\in\Lie\Aut(\Sigma,L)$ admits a
unique decomposition
\begin{equation}\label{eq:y*}
\zeta=\zeta_1+I\zeta_2,
\end{equation}
where $\zeta_1,\zeta_2$ are in the Lie algebra of the extended gauge
group of $(\omega,h)$. Furthermore, any $(\omega,h)\in B$ determines a
space $\cT$ with a moment map $\mu_\alpha$ as in
Proposition~\ref{prop:momentmap-inttriples}, and using the
decomposition~\eqref{eq:y*} and a change of variable in
\eqref{eq:sigma-explicit2} (cf.~\cite[(3.104)]{AGG}), we obtain
\begin{equation}\label{eq:sigmamu}
\sigma_{\alpha,\tau}(\dot b_t)=\langle \mu_{\alpha}(J,A,g_t\cdot\phi),\zeta_2\rangle,
\end{equation}
where $g_t\cdot(J,A,\phi)=(J,A,g_t\cdot\phi)\in\mathcal{T}$, as
$g_t\in\Aut(\Sigma,L)$. We observe now that the proof
of~\cite[Proposition 3.8]{AGG} works for the 1-form
$\sigma_{\alpha,\tau}$ (which does depend on $\phi$), so
\begin{equation}\label{eq:sigmainc}
\frac{d}{dt}\sigma_{\alpha,\tau}(\dot b_t) = \|Y_{\zeta_2|(J,A,g_t\cdot\phi)}\|^2
\geq 0
\end{equation}
(cf.~\cite[(3.113)]{AGG}), where $Y_{\zeta_2|(J,A,g_t\cdot\phi)}$
denotes the infinitesimal action of $\zeta_2$ on
$(J,A,g_t\cdot\phi)\in\cT$, and the norm is given by the K\"ahler form
on $\mathcal{T}$ described in Section~\ref{sec:mmap} --- it is a
positive definite norm precisely because $\alpha>0$ (note
that~\eqref{eq:ddtsigma} follows from~\eqref{eq:sigmainc}).
The proof of the lemma is now straightforward from~\eqref{eq:sigmamu}
and the definition of the Futaki invariant~\eqref{eq:futakigravvort}.
\end{proof}

\begin{proof}[Proof of Theorem~\ref{th:Yangconjectureintro}]
We follow the notation in the proof of Theorem~\ref{th:Yangconjecture.2}.
%
Assume that $\phi$ is not polystable. Then there exists a
$1$-parameter subgroup
\[
\lambda\colon\CC^*\lto\SL(2,\CC)\subset\Aut(\PP^1,L),
\]
and a suitable choice of homogeneneous coordinates, such that (up to
rescaling)
\[ 
\lim_{t \to + \infty} \lambda(e^{-t}) \cdot \phi = x_0^{N- \ell}x_1^\ell,
\]
with non-positive Hilbert--Mumford weight, that is,
\begin{equation}\label{eq:unstable-configuration}
N-2\ell\leq 0
\end{equation}
(cf. e.g. the proof of~\cite[Theorem~3.10]{ThomasGIT}). 
Consider the geodesic ray $b_t=\lambda(e^t)^*(\omega,h)$. The
corresponding maximal weight is
\[
w(\PP^1,L,\phi)=\operatorname{Im}\,\langle\cF_{\alpha,\tau}(\phi_0),\zeta\rangle,
\]
by Lemma~\ref{lem:weight-P1}, with
\begin{align*}
\zeta=\(\begin{array}{cc}
N-2\ell-1 & 0 \\
0 & N-2\ell+1
\end{array}\),
\end{align*}
and therefore Lemma~\ref{lem:evaluationfutaki} implies
\begin{equation}\label{eq:maximalweight-P1-configuration}
w(\PP^1,L,\phi)=4\pi\alpha(\tau-2N)(N-2\ell).
\end{equation}
Assume now $(\PP^1,L,\phi)$ admits a solution $(\omega,h)$ of the
gravitating vortex equations, that is, $\sigma_{\alpha,\tau}$ vanishes
at $(\omega,h)$. Then $\sigma_{\alpha,\tau}(\dot b_0)=0$
and~\eqref{eq:ddtsigma} imply that for any geodesic ray,
\begin{equation}\label{eq:semistable-configuration}
w(\PP^1,L,\phi)\geq 0.
\end{equation}
Moreover, Theorem~\ref{th:B-GP} implies
\begin{equation}\label{eq:vortices-conf-points}
\tau-2N>0.
\end{equation}
However, if the inequality~\eqref{eq:unstable-configuration} is
strict, i.e.  $N-2\ell<0$, then $w(\PP^1,L,\phi)<0$,
by~\eqref{eq:maximalweight-P1-configuration}
and~\eqref{eq:vortices-conf-points},
contradicting~\eqref{eq:semistable-configuration}. In the remaining
case $N-2\ell=0$, $w(\PP^1,L,\phi)=0$
by~\eqref{eq:maximalweight-P1-configuration}, which combined
with~\eqref{eq:sigmainc} and $\sigma_{\alpha,\tau}(\dot b_0)=0$, imply
that $\zeta$ fixes $\phi$, so $\phi=\phi_0$, but this cannot happen
because $\phi_0$ is polystable.
\end{proof}

\subsection{A conjecture about uniqueness and the moduli space of
  gravitating vortices}\label{sec:conjecturemoduli}

The contents of Sections~\ref{sec:geodesicstability}
and~\ref{subsec:geodesics}, especially
Proposition~\ref{prop:geod_convex-uniqueness}, suggest an approach to
the uniqueness problem for the gravitating vortex equations on a
compact Riemann surface $\Sigma$ of arbitrary genus. To the knowledge
of the authors, this problem has not been explored so far, even for
the Einstein--Bogomol'nyi equations (for which $\Sigma=\PP^1$). This
approach rests on the geometry of the infinite-dimensional space $B$,
and the closely related space $\mathcal{K}$ of K\"ahler forms on
$\Sigma$ with fixed volume $\Vol(\Sigma)$. The space $B$ is a
symmetric space, as briefly reviewed in
Section~\ref{sec:geodesicstability}, and the space $\mathcal{K}$ is a
Riemannian symmetric space, as shown by Semmes~\cite{Se} and
rediscovered by Mabuchi~\cite{Mab1} and Donaldson~\cite{D6}. Since the
geodesic equation on $\mathcal{K}$ is the first equation
in~\eqref{eq:geodesicequation2}, i.e. the map
\[
B\longrightarrow\mathcal{K},
\]
given by $(\omega,h)\longmapsto\omega$, is a geodesic submersion, one
cannot expect in general existence of smooth geodesic segments on $B$
with arbitrary boundary conditions, by results of Lempert and
Vivas~\cite{LeVi} about the geometry of $\mathcal{K}$. Hence one
cannot expect either that a direct application of
Proposition~\ref{prop:geod_convex-uniqueness} will work in the
uniqueness problem for the gravitating vortex equations. Here we
propose a possible way to circumvent this difficulty.

As shown by Donaldson~\cite{D6} and Semmes~\cite{Se}, for a suitable
choice of Riemann surface $D$, the geodesic equation on $\mathcal{K}$
reduces to a homogeneous complex Monge--Amp\`ere equation on the
complex surface $\Sigma \times D$. This method has been fruitfully
applied in the context of the problem for constant scalar curvature
K\"ahler metrics~\cite{Blo,Ch1,ChT,JN}. We expect that these results,
and in particular the recent proof of the uniqueness of constant
scalar curvature K\"ahler metrics by Berman and Berndtsson
\cite{BermanBerndtsson}, can be adapted to the context of the geodesic
equation~\eqref{eq:geodesicequation2}. In light of this, the following
conjecture seems reasonable.

\begin{conjecture}\label{conj:2}
Given $\alpha > 0$, if $\tau$ satisfies \eqref{eq:ineq} and $\phi$ is
polystable, then there exists a unique solution of the gravitating
vortex equations on $(\PP^1,L,\phi)$ modulo automorphisms.
\end{conjecture}

A proof of Conjecture \ref{conj:2}, combined with Theorem \ref{th:HK},
would lead to the following explicit description of the moduli space
of solutions of the Einsten--Bogomol'nyi equations.

\begin{conjecture}\label{conj:3}
The moduli space of solutions of the degree-$N$ Einstein--Bogomol'nyi
equations is biholomorphic to the GIT quotient 
\begin{equation}\label{eq:GITquotient}
S^N \PP^1 /\!\!/ \operatorname{SL}(2,\CC).
\end{equation}
\end{conjecture}

Furthermore, it is reasonable to hope that Yang's Existence
Theorem~\ref{th:Yangintro} for the Einstein--Bogomol'nyi equations
holds for the more general gravitating vortex equations on $\PP^1$ in
the case $\alpha>0$. This result, combined with
Conjecture~\ref{conj:2}, would provide an explicit description of the
moduli space of gravitating vortices on $\PP^1$, exactly as in
Conjecture~\ref{conj:3}.

Note that the biholomorphism of Conjecture \ref{conj:3} would show an
intriguing link between the physics of cosmic strings and the
classical theory of binary quantics \cite{MFK,Sylvester}.


\section{Existence and uniqueness of gravitating vortices in $g\geq 2$}
\label{sec:higher-genus}

In this section we prove that the gravitating vortex
equations~\eqref{eq:gravvortexeq1} have a unique solution in genus
$g\geq 2$, assuming a suitable effective bound on the coupling
constant $\alpha > 0$ (provided that the inequality~\eqref{eq:ineq} is
satisfied). The main result of this section, combined with Theorem
\ref{th:HK}, draws a parallel between the existence problem for the
gravitating vortex equations and the K\"ahler--Einstein problem, where
stability only plays a role in the Fano case
\cite{Aubin,ChDoSun,Yau1977,Yau} (i.e. positive canonical bundle).


\subsection{Statement of the result and the continuity method}
\label{sub:openness}

\begin{theorem}\label{thm:higher-genus}
Let $\Sigma$ be a compact Riemann surface of genus $g\geq 2$, and $L$
a holomorphic line bundle over $\Sigma$ of degree $N>0$ equipped with a
holomorphic section $\phi \neq 0$. Let $\tau$ be a real constant such
that $0<N<\tau/2$. Define
\begin{equation}
\label{eq:critical-alpha.2}
\alpha_*\defeq\frac{2g-2}{2\tau(\tau/2-N)}>0.  
\end{equation}
Then the set of $\alpha\geq 0$ for which~\eqref{eq:gravvortexeq1} has
smooth solutions of volume $2\pi$ is open and contains the closed
interval $[0,\alpha_*]$. Furthermore, the solution is unique for
$\alpha\in[0,\alpha_*]$.
\end{theorem}

For the purposes of the proof, 
we fix a metric $\omega_0$ of constant curvature $-1$ on $\Sigma$ with
volume $2\pi$, a Hermite--Einstein metric $h_0$ on $L$ (with respect
to $\omega_0$), and define
\begin{equation*}\label{constants}
c_{\alpha}=\chi-2\alpha\tau N\leq 0,
\end{equation*}
for $\alpha\in[0,\alpha_{*}]$ as in~\eqref{eq:constantc}, where $\chi=2-2g<0$ is the Euler
characteristic of $\Sigma$, so~\eqref{eq:critical-alpha.2} is
equivalent to
\begin{equation}\label{critical-alpha}
0=c_{*}+\alpha_{*} \tau^2=\chi+\alpha_{*}\tau\left(\tau -2N\right).
\end{equation}
Furthermore, we will consider the equations~\eqref{eq:KWtype0}, with
unknowns $u,f\in C^\infty(\Sigma)$, which are equivalent to the
gravitating vortex equations~\eqref{eq:gravvortexeq1}, for
$\omega=\omega_0+ dd^c v$ and $h=h_0e^{2f}$. We will use the method of
continuity, where the continuity path is simply
equations~\eqref{eq:KWtype0}, with the coupling constant $\alpha$
replaced by a continuity parameter $t$, i.e.,
\begin{equation}\label{main-equations-with-psi-equal-to-zero-continuity-path}
\begin{split}
\Delta_0 f + \frac{1}{2} (\vert \phi \vert_h ^2 -\tau) e^{4t \tau f-2t \vert \phi \vert_h^2 -2cv} 
&= -N ,
\\
\Delta_0 v + e^{4t \tau f-2t \vert \phi \vert_h^2 -2cv} & =1,
\end{split}\end{equation}
where $t\in [0,\alpha_{*}]$, with
\begin{align*}&
\vert\phi\vert_h^2=\vert\phi\vert_{h_0}^2 e^{2f},
\\&
c=c_t\defeq \chi-2t\tau N.
\end{align*}

At $t=0$, the
equations~\eqref{main-equations-with-psi-equal-to-zero-continuity-path}
decouple. First, one can solve the second equation, as it becomes the
constant curvature condition in this case, and then the first equation
is the vortex equation, which is solved by results of
Noguchi~\cite{Noguchi}, Bradlow~\cite{Brad} and the third
author~\cite{G1,G3} (see Theorem~\ref{th:B-GP}).
To apply the method of continuity, using an implicit function theorem
argument, we now prove that set $S\subset\R_{\geq 0}$ of all $t$
such~\eqref{main-equations-with-psi-equal-to-zero-continuity-path} has
a smooth solution $(f_t,v_t)$ (which is denoted $(f,v)$ abusing
notation) depending smoothly on $t$ is open.

\begin{remark}\label{kahler-preservation}
Note that $\omega=\omega_0 + dd^c v = (1- \Delta v) \omega_0$ and therefore the second equation of~\eqref{main-equations-with-psi-equal-to-zero-continuity-path}
automatically implies that $\omega$ is K\"ahler.
\end{remark}

\begin{lemma}\label{lem:openness}
The subset $S\subset\R$ is open.
\end{lemma}

\begin{proof}
As observed in Section~\ref{subsec:gravvort}, if $(v,f)$ is a solution
of~\eqref{main-equations-with-psi-equal-to-zero-continuity-path}, then
$(\omega_0 + dd^c v,e^{2f}h_0)$ is a solution of \eqref{eq:KEtype}
(and hence of \eqref{eq:gravvortexeq1}) for $\alpha = t$, because
\eqref{eq:KEtype} is equivalent to
\begin{equation}\label{eq:lemmaopen1}
\begin{split}
\Delta f + \frac{1}{2}(e^{2f}|\phi|^2-\tau)(1- \Delta v) & = - N,\\
dd^c \(\log (1-\Delta v) - 4\alpha \tau f + 2 \alpha e^{2f}|\phi|^2 + 2 c v\) & = 0.
\end{split}
\end{equation}
Conversely, for any solution $(\omega_0 + dd^c v,e^{2f}h_0)$
of~\eqref{eq:gravvortexeq1} with $\alpha = t$, it follows that
\[
\log (1-\Delta v) - 4\alpha \tau f + 2 \alpha e^{2f}|\phi|^2 + 2 c v = c'
\]
by~\eqref{eq:lemmaopen1}, for some constant $c' \in \RR$, and
therefore there exist unique constants $c_0,c_1$ (that depend
continuously on $f,v$ in $C^0$-norm), such that $f + c_0$ and $v +
c_1$ solve
\eqref{main-equations-with-psi-equal-to-zero-continuity-path}. Hence
it suffices to prove that the set of $\alpha \in \RR_{\geq 0}$ for which there
exists a solution of \eqref{eq:gravvortexeq1} is open. To prove this
fact, we use the properties of the moment map of
Proposition~\ref{prop:momentmap-inttriples}. Consider the linear
differential operator
\begin{equation}
\label{eq:Lalphaoperator}
\mathbf{L}_{\alpha}= (\mathbf{L}_{\alpha}^0,\mathbf{L}_{\alpha}^1) \colon C^\infty(X) \times C^\infty(X) \lto C^\infty(X) \times C^\infty(X),
\end{equation}
such that the value at $(\dot v, \dot f)$ is given by the
linearization 
\[
(\mathbf{L}_{\alpha}^0,\mathbf{L}_{\alpha}^1) = \delta \Bigg{(}i\Lambda_\omega F_h + \frac{1}{2}|\phi|^2_h - \frac{\tau}{2}, - S_\omega - \alpha \Delta_\omega |\phi|^2_h + 2\alpha\tau i\Lambda_\omega F_h + c\Bigg{)}
\]
of the moment map at $(J,A,\phi)$ along the infinitesimal action of
the vector field
\[
A^\perp J \eta_{\dot v} - \dot f.
\]
Here, $J$ is the almost complex structure on $\Sigma$, $A$ is the
Chern connection of $h$ on the line bundle $L$ and $\eta_{\dot v}$ is
the hamiltonian vector field of $\dot v$ with respect to
$\omega$. More explicitly,
\begin{equation}
\label{eq:Dalphapaoperator}
\begin{split}
\mathbf{L}_{\alpha}^0(\dot v,\dot f) & = d^*(d \dot f + \eta_{\dot v}\lrcorner i F_h) + (\phi, -J \eta_{\dot v}\lrcorner d_A \phi + 2 f \phi)_h,\\
\mathbf{L}_{\alpha}^1(\dot v,\dot f) & = \operatorname{P}^*\operatorname{P} \dot v - 4\alpha i d (d_A \phi , -J \eta_{\dot v}d_A \phi + 2 f \phi)_h \\
& - 2 \alpha i d ((d \dot f + \eta_{\dot v}\lrcorner i F_h)|\phi|_h) + 2\alpha \tau d^*(d \dot f + \eta_{\dot v}\lrcorner i F_h),
\end{split}
\end{equation} 
where $\operatorname{P}^*\operatorname{P}$ is, up to a multiplicative
constant factor, the Lichnerowicz operator of the K\"ahler manifold
$(\Sigma,J,\omega)$ \cite{Lichnerowicz}. Consider now the moment-map
operator
\begin{equation}
\label{eq:Talphaoperator}
\mathbf{T}_{\alpha}= (\mathbf{T}_{\alpha}^0,\mathbf{T}_{\alpha}^1)
\colon C^\infty(X) \times C^\infty(X) \lto C^\infty(X) \times C^\infty(X),
\end{equation}
given by 
\begin{align*}
\mathbf{T}_{\alpha}^0(v,f) & = i\Lambda_\omega F_h + \frac{1}{2}|\phi|^2_h - \frac{\tau}{2},\\
\mathbf{T}_{\alpha}^1(v,f) & = - S_\omega - \alpha \Delta_\omega |\phi|^2_h + 2\alpha\tau i\Lambda_\omega F_h + c,
\end{align*}
where $(\omega,h) = (\omega_0 + dd^c v,e^{2f}h_0)$. Arguing as in
\cite[Proposition 4.7]{AGG} for our moment map \eqref{eq:mutriples},
the linearization of $\mathbf{T}_{\alpha}$ at $(v,f)$ satisfies
\begin{align*}
\delta \mathbf{T}_{\alpha}^0(v,f) & = \mathbf{L}^0_\alpha(\dot v,\dot f) + J\eta_{\dot{v}}\lrcorner d(\mathbf{T}_{\alpha}^0(v,f)),\\
\delta \mathbf{T}_{\alpha}^1(v,f) & = \mathbf{L}^1_\alpha(\dot v,\dot f) + (d(\mathbf{T}_{\alpha}^1(v,f)),d\dot v)_\omega,
\end{align*}
and so $\delta \mathbf{T}_{\alpha} = \mathbf{L}_\alpha$ if
$(\omega,h)$ is a solution of \eqref{eq:gravvortexeq1}. Moreover, as
in~\cite[Lemma 4.5]{AGG}, we have
\begin{align*}
\langle (\dot v, 4\alpha \dot f),\mathbf{L}_\alpha (\dot v, \dot f) \rangle_{L^2} & = \|L_{\eta_{\dot v}}J\|_{L^2}^2 + 4\alpha  \|d \dot f + \eta_{\dot v}\lrcorner i F_h\|_{L^2}^2 + 4\alpha\|J \eta_{\dot v}\lrcorner d_A \phi - 2 f \phi \|^2_{L^2},\\
& + 4\alpha \langle (J \eta_{\dot v} \lrcorner (i d \dot f + \eta_{\dot v}\lrcorner F_h),\mathbf{T}_{\alpha}^0(v,f) \rangle_{L^2}.
\end{align*}
Assuming that $(v,f)$ is a solution of \eqref{eq:gravvortexeq1}, that
is, $\mathbf{T}_{\alpha}(v,f) = 0$, the operator $\delta
\mathbf{T}_{\alpha}$ is self-adjoint and the condition
$$
\delta \mathbf{T}_{\alpha}(\dot v, \dot f) = 0
$$
implies that $A^\perp \eta_{\dot v} + i \dot f$ is an infinitesimal
automorphism of $(\Sigma,L,\phi)$. By assumption, $g\geq 2$ and
$\phi\neq 0$, and by Proposition
\ref{prop:automorphism-groups.positive-genus} it follows that $\dot v,
\dot f$ must be constant functions on $\Sigma$. The result now follows
by application of the implicit function theorem in a Sobolev
completion of $C^\infty(X) \times C^\infty(X)$.
\end{proof}

\subsection{Closedness, existence, and uniqueness}
\label{sub:existence}

In this section, we prove that the set of $t\in[0,\alpha_{*}]$ for which
the system admits a smooth solution is also closed. By the previous
section, this means that the set is all of $[0,\alpha_{*}]$, thus proving
the existence part of Theorem \ref{thm:higher-genus}. To do so, we
prove $C^{2,\gamma}$ \emph{a priori} estimates for $(f,v)$ independent
of $t$. The Arzela--Ascoli theorem with usual elliptic bootstrapping
implies closedness.

Firstly, we reduce the problem to proving a $C^0$ estimate.

\begin{proposition}
Suppose $(f,v)$ is a smooth solution of \eqref{main-equations-with-psi-equal-to-zero-continuity-path}. Assume that there exists $C > 0$ independent of $t$ such that $\Vert v \Vert_{C^0}\leq C, \Vert f \Vert _{C^{0}}\leq C$. Then, for some $\gamma>0$, $\Vert v \Vert_{C^{2,\gamma}}$ and $\Vert f \Vert_{C^{2,\gamma}}$ are bounded independently of $t$.
\label{reduction-to-uniform-estimate}
\end{proposition} 

\begin{proof}
Indeed, rewriting equations \eqref{main-equations-with-psi-equal-to-zero-continuity-path} as
\begin{equation}\label{rewrittten-equations}
\begin{split}
\Delta_0 f =- \frac{1}{2} (\vert \phi \vert_h ^2 -\tau) e^{4t \tau f-2t \vert \phi \vert_h^2 -2cv}  -N,\\
\Delta_0 v =- e^{4t \tau f-2t \vert \phi \vert_h^2 -2cv} + 1,
\end{split}\end{equation}
we see that the right-hand sides have $L^p$-norm bounds independent of $t$ for all $p>0$. Therefore
by elliptic regularity we see that $v$ and $f$ have $W^{2,p}$-norm bounded independently of $t$ and hence
$C^{1,\gamma}$-norm bounded independently of $t$. This means that the right-hand sides have $C^{0,\gamma}$-norm bounds independent of $t$ and the Schauder estimates allow us to conclude the proof.
\end{proof}

Secondly, we prove the following useful inequality for solutions of
the vortex equation \eqref{eq:vortexeq}.

\begin{lemma}\label{useful-inequality-on-phi}
Assume that $(\omega,h)$ is a solution of the vortex equation \eqref{eq:vortexeq}. Then, 
$$
\vert \phi\vert_{h}^2 - \tau \leq 0.
$$
\end{lemma}

\begin{proof}
Let $b=\vert \phi\vert_{h}^2 - \tau$. At the maximum of $b$, $\nabla b=0$. Computing the laplacian at the maximum, by the Weitzenb\"ock formula we see that
\begin{gather}
0\leq \Delta_0 b \leq i \Lambda_\omega F_h \vert \phi \vert_h ^2.
\label{laplacian-of-b-at-max}
\end{gather}
Using now~\eqref{eq:vortexeq}, we see that indeed
$\vert\phi\vert_{h}^2-\tau\leq 0$.
\end{proof}

From now onwards we assume that $(f,v)$ is a smooth solution to \eqref{main-equations-with-psi-equal-to-zero-continuity-path} unless specified otherwise. We also denote all constants independent of $t$ as $C$ by default. Note that the equations \eqref{main-equations-with-psi-equal-to-zero-continuity-path} force a normalisation condition on $(f,v)$ via integration. This allows us to prove an integral bound:

\begin{lemma}\label{integral-bounds}
\begin{equation}\label{integral-bounds-equation}
\begin{split}
\displaystyle \int_\Sigma\left((2+4t\tau)f-2cv\right) \omega_0 \leq C, \\
\displaystyle \int_\Sigma\left(4t\tau f-2cv\right) \omega_0 \leq C.
\end{split}\end{equation}
\end{lemma}

\begin{proof}
Integrating
\eqref{main-equations-with-psi-equal-to-zero-continuity-path} after
multiplying by $\omega_0$, we see that  
\begin{equation}\label{before-using-Jensen}
\begin{split}
\displaystyle \int_\Sigma(\vert \phi \vert_h^2 - \tau) e^{4t\tau f-2t\vert
  \phi \vert_h^2 - 2cv} \omega_0 = -4\pi N , \\
\displaystyle \int_\Sigma e^{4t\tau f-2t\vert \phi \vert_h^2 - 2cv} \omega_0 = 2\pi.
\end{split}\end{equation}
Using the second equation in the first, we get
\begin{align*}
\displaystyle \int_\Sigma e^{\ln(\vert \phi \vert_{h_0}^2)+(2+4t\tau) f-2t\vert \phi \vert_h^2 - 2cv} \omega_0 &= 2\pi(\tau-2N), \\
\displaystyle \int_\Sigma e^{4t\tau f-2t\vert \phi \vert_h^2 - 2cv} \omega_0 &= 2\pi.
\end{align*}
We use Jensen's inequality $e^{\fint f} \leq \fint e^f$ to conclude
\begin{align*}
e^{\fint_\Sigma [\ln(\vert \phi \vert_{h_0}^2)+(2+4t\tau) f-2t\vert \phi
  \vert_h^2 - 2cv]} &\leq C, \\
e^{\fint_\Sigma [4t\tau f-2t\vert \phi \vert_h^2 - 2cv]} &\leq C.
\end{align*}
Since $\phi$ is locally of the form $z^k$, $\ln \vert \phi \vert_{h_0}^2$ is integrable. Moreover, $0\leq \vert \phi \vert_h ^2 \leq \tau$. Hence this implies the desired result.  
\end{proof}

Next we consider the function $y=e^{4t\tau f-2cv}$ and prove a bound on it.

\begin{lemma}
Assuming that $\chi < 0$, then $-C\leq \ln(y)$. Furthermore, if $c+t\tau^2 \leq 0$, then $-C\leq \ln(y) \leq C$, and therefore 
$\Vert 4t\tau f -2cv\Vert_{C^{0}} \leq C$.
\label{bounds-on-y}
\end{lemma}

\begin{proof}
At an extremum of $y$, $\nabla y=0$. Let $p$ be such a point of extrema. We compute
\begin{align}
\Delta_0 y (p) & = y(p)[4t\tau \Delta_0 f(p) - 2c \Delta_0 v(p)]
\nonumber\\&
\Longrightarrow \frac{\Delta_0 y(p)}{y(p)} = 4t \tau \left(  -\frac{1}{2}(\vert\phi \vert_h^2 -\tau)y(p)e^{-2t\vert \phi \vert_h^2} -N\right) -2c \(1-y(p) e^{-2t\vert \phi \vert_h^2}\)
\nonumber\\&
\Longrightarrow  \frac{\Delta_0 y(p)}{2e^{-t\vert\phi\vert_h ^2}y(p)} = y(p) [c-t\tau (\vert \phi \vert_h^2 -\tau) ]-(c+2t\tau N)e^{2t\vert \phi \vert_h^2}.
\label{laplacian-of-y-at-critical-point}
\end{align}

At a point of minimum, $\Delta_0 y(p) \leq 0$. Hence, 
\begin{align}
y_{min}[c-t\tau &(\vert \phi \vert_h^2 -\tau) ]\leq (c+2t\tau N)e^{2t\vert \phi \vert_h^2} \leq (c+2t\tau N)  = \chi <0
\nonumber\\&
\label{minimum}
\Longrightarrow y_{min} c \leq \chi. 
\end{align}
Thus $y_{min} \geq \frac{1}{C}$. This means that $\ln(y) \geq -C$.

Now, the Green representation formula (see \cite[Theorem 4.13,
p. 108]{Aubinbook}) is
\begin{gather}
u(P) = \fint_\Sigma u(Q) \omega_0 (Q) + \displaystyle \int_\Sigma G(P,Q) \Delta_0 u (Q) \omega_0 (Q),
\label{Greenrep}
\end{gather}
where $0\leq G(P,Q)\leq C+C\vert \ln (d_{\omega_0}(P,Q)) \vert$. Applying this to $u=\ln(y)=4t\tau f -2cv$, we see that
\begin{gather}
\ln(y) = \fint_\Sigma (4t\tau-2cv) \omega_0 +\displaystyle \int_\Sigma G(P,Q) (2y[c-t\tau(\vert \phi \vert_h^2 -\tau)]e^{-t\vert \phi \vert_h^2}-2(c+2t\tau N)) \omega_0 \nonumber \\
\leq C +C\displaystyle \int_\Sigma G(P,Q) y[c+t\tau^2] \omega_0 \leq C,
\label{After-applying-Greenrep}
\end{gather}
where the last equality follows from the assumption $c+t\tau^2\leq 0$
(cf.~\eqref{critical-alpha}). Thus $-C\leq \ln(y) \leq C$.
\end{proof}

Next we prove one-sided bounds.

\begin{lemma}
$f \leq C$ and $v \geq -C$.
\label{one-sided-C0-estimate}
\end{lemma}

\begin{proof}
The estimate in Lemma \ref{bounds-on-y} shows that $-C\leq 4t\tau f -2cv \leq C$. This means that 
\begin{gather}
\displaystyle -C \leq \int_\Sigma (4t\tau f -2cv) \omega_0 \leq C.
\label{integral-bound-on-lny}
\end{gather}
This in conjunction with the first inequality in
\eqref{integral-bounds-equation} implies that $\int_\Sigma f\leq
C$. Using this along with inequality \eqref{integral-bound-on-lny}
implies that $\int_\Sigma v \geq -C$. By the Green representation
formula \eqref{Greenrep}, this means that $v\geq -C$ and hence $f\leq
C$.
\end{proof}

This means that the function $0<U=e^{2f} \leq C$. Define $\tilde{f} =
f-\fint_\Sigma f$ and $\tilde{v} = v-\fint_\Sigma v$. We now prove
bounds on these two functions.

\begin{lemma}
$\Vert \tilde{f} \Vert_{C^{1,\gamma}} \leq C$, $\Vert \tilde{v} \Vert_{C^{1,\gamma}} \leq C$.
\label{bounds-on-averagefree-parts}
\end{lemma}

\begin{proof}
Writing equations \eqref{rewrittten-equations} in terms of $\tilde{f}, \tilde{v}$, we get
\begin{equation}\label{equations-for-averagefree-parts}
\begin{split}
\Delta_0 \tilde{f} = \Delta_0 f = - \frac{1}{2} (\vert \phi \vert_h ^2 -\tau) e^{4t \tau f-2t \vert \phi \vert_h^2 -2cv}  -N = \eta_1,  \\
\Delta_0 \tilde{v} =\Delta_0 v =- e^{4t \tau f-2t \vert \phi \vert_h^2 -2cv} + 1 =\eta_2.
\end{split}\end{equation}
By previous arguments, $\eta_1$ and $\eta_2$ are bounded above and
below independent of $t$. Multiplying the first equation of
\eqref{equations-for-averagefree-parts} by $\tilde{f}$, the second one
by $\tilde{v}$, and integrating by parts, we get
\begin{equation}\label{before-using-poincare} 
\begin{split}
\displaystyle \int_\Sigma \vert \nabla _0 \tilde{f} \vert ^2 \omega_0 = \int_\Sigma \eta_1 \tilde{f} \omega_0\leq \int_\Sigma \frac{1}{2\epsilon} \eta_1 ^2\omega_0 + \frac{\epsilon}{2}\int_\Sigma \tilde{f}^2 \omega_0 , \\
\displaystyle \int_\Sigma \vert \nabla _0 \tilde{v} \vert ^2 \omega_0 = \int_\Sigma \eta_2 \tilde{v} \omega_0 \leq \int_\Sigma \frac{1}{2\epsilon} \eta_2 ^2\omega_0 + \frac{\epsilon}{2}\int_\Sigma \tilde{v}^2 \omega_0.
\end{split}\end{equation}
Since $\tilde{f}$ and $\tilde{v}$ have zero average, one can use Poincar\'e's inequality $\Vert \tilde{f} \Vert_{L^2} \leq C \Vert \nabla_0 \tilde{f}\Vert_{L^2}$ in \eqref{before-using-poincare} to conclude that
\begin{equation}\label{W12}
\begin{split}
\displaystyle \int_\Sigma \tilde{f} ^2 \omega_0 + \int_\Sigma \vert \nabla _0 \tilde{f} \vert^2 \omega_0 \leq C, \\
\displaystyle \int_\Sigma \tilde{v}^2 \omega_0 + \int_\Sigma \vert \nabla _0 \tilde{v} \vert^2 \omega_0 \leq C. 
\end{split}\end{equation}

Using \ref{W12} and the Sobolev embedding theorem, we conclude
that the $L^p$ norms of $\tilde{f}$ and $\tilde{v}$ are bounded for
all $p$. Going back to \ref{equations-for-averagefree-parts}, we see
using $L^p$ elliptic regularity that $\Vert \tilde{f}\Vert_{W^{2,p}}
\leq C$ and $\Vert \tilde{v} \Vert_{W^{2,p}}\leq C$. Thus the Sobolev
embedding theorem gives the desired result.
\end{proof}

In the next result we prove the desired $C^0$-estimate for $f,v$. By Proposition \ref{reduction-to-uniform-estimate} this implies $C^{2,\gamma}$ bounds on $f,v$.

\begin{proposition}\label{prop:C0}
Suppose $(f,v)$ is a smooth solution of \eqref{main-equations-with-psi-equal-to-zero-continuity-path}. Then $\Vert v \Vert_{C^0}\leq C, \Vert f \Vert _{C^{0}}\leq C$ independent of $t$.
\label{C0-estimate}
\end{proposition} 
\begin{proof}
Since $\vert \nabla f \vert = \vert \nabla \tilde{f} \vert \leq C$,
this means that $0 \leq \max f - \min f\leq C$. If $\max f$ is not
bounded below, then there exists a sequence $t_n \rightarrow t$ such
that $f_n \rightarrow -\infty$ everywhere. This would mean that $\vert
\phi \vert_h^2 \rightarrow 0$ everywhere. By the dominated convergence
theorem applied to the first equality in \eqref{before-using-Jensen}, we
find a contradiction because $\tau >2N$ by assumption. Therefore, $-C
\leq f \leq C$. This implies that $-C \leq v \leq C$ because $-C\leq
\ln y \leq C$.
\end{proof} 

\emph{Proof of Theorem \ref{thm:higher-genus}.} 
The existence part follows from Proposition \ref{prop:C0} and Proposition \ref{reduction-to-uniform-estimate}. As for uniqueness, suppose there are two solutions $(f_1,v_1)$
and $(f_2,v_2)$. We will use a method inspired by Bando and Mabuchi
\cite{BM}. Run the continuity method backwards starting at $t=\alpha$
for both of these solutions. Openness holds as long as $t \geq 0$. The
\emph{a priori} estimates also hold as long as $t\geq 0$. We know that
at $t=0$, the system decouples and the solution is unique. Suppose $T$
is the supremum of all $t$ such that there is a unique solution on
$[0,T]$, i.e., both continuity paths agree on $[0,T]$. By the \emph{a priori} estimates, the continuity paths of $(f_1,v_1)$ and $(f_2,v_2)$ are defined on
$t\in(T-\epsilon,T+\epsilon)$ and their points lie near the unique
solution for $t=T$ (in the $C^{2,\gamma}$ topology).  Furthermore, by
openness, there is a unique solution near the unique solution for
$t=T$ (in the $C^{2,\gamma}$ topology). Thus, uniqueness holds beyond
$T$ and that is a contradiction unless $(f_1,v_1)=(f_2,v_2)$.

\qed

We finish this section with an application of Theorem~\ref{thm:higher-genus.intro} to construct a new class of non-trivial solutions of the K\"ahler--Yang--Mills equations, introduced in~\cite{AGG}. Recall that the gravitating vortex equations were obtained in~\cite{AGG2} as a dimensional reduction of the K\"ahler--Yang--Mills equations. In this context, they are defined on a pair $(X,E)$ consisting of the product
$X=\Sigma\times\PP^1$ with K\"ahler form $\omega_\tau=p^*\omega
+\frac{4}{\tau}q^*\omega_{FS}$, where $\omega_{FS}$ is the
Fubini--Study metric on the complex projective line $\PP^1$, $p\colon
X\to\Sigma$ and $q\colon X\to\PP^1$ are the canonical projections, and
$E$ is the rank 2 vector bundle over $X$ given by the holomorphic
extension in $\Ext^1_X(q^*\cO_{\PP^1}(2),p^*L)\cong H^0(\Sigma,L)$
determined by $\phi$ (cf.~\cite[Theorem~1.1]{AGG2}).

\begin{theorem}\label{thm:dim-red}
Let $\Sigma, L, \phi, N,\tau$ and $\alpha_*$ be as in
Theorem~\ref{thm:higher-genus.intro} (so $\Sigma$ has genus $g\geq
2$). Then the set of $\alpha\geq 0$ for which the
K\"ahler--Yang--Mills equations on $(\Sigma\times\PP^1,E)$, with
K\"ahler form $\omega_\tau$, admit an $\SU(2)$-invariant solution of
fixed volume, is open and contains the closed interval
$[0,\alpha_*]$. Furthermore, the $\SU(2)$-invariant solution is unique
for $\alpha\in[0,\alpha_*]$.
\end{theorem}

\begin{proof}
Immediate by Theorem~\ref{thm:higher-genus.intro} and~\cite[Proposition~3.4]{AGG2}.
\end{proof}

\end{document}